\documentclass[letter,11pt]{article}

\usepackage[margin=1.2 in, top=1 in, bottom= 1.2 in]{geometry}

\usepackage[normalem]{ulem}

\usepackage{tikz}
\usetikzlibrary{hobby}
\usepackage{tikz-cd}
\usepackage{adjustbox}
\usepackage{fbox}

\usetikzlibrary{matrix, calc, arrows}

\usepackage{bibentry}
\usepackage[linktocpage=true]{hyperref}

\usepackage{lipsum}
\newcommand\blfootnote[1]{%
  \begingroup
  \renewcommand\thefootnote{}\footnote{#1}%
  \addtocounter{footnote}{-1}%
  \endgroup
}

\makeatletter
\g@addto@macro\normalsize{
  \setlength\abovedisplayskip{8pt}
  \setlength\belowdisplayskip{8pt}
  \setlength\abovedisplayshortskip{8pt}
  \setlength\belowdisplayshortskip{8pt}
  }
\makeatother

\usepackage{tocloft}

\usepackage[english]{babel}
\usepackage{amsfonts}
\usepackage{mathrsfs}
\usepackage{bbm}
\usepackage{latexsym}
\usepackage{math dots}
\usepackage{amssymb}
\usepackage{mathtools}

\usepackage{enumitem}
\setlist{nolistsep}
\usepackage{amsthm}
\usepackage[capitalize]{cleveref}
\crefformat{footnote}{#2\footnotemark[#1]#3} 

\newcommand\eqnitem[1][]{%
  \ifx\relax#1\relax  \item \else \item[#1] \fi
  \abovedisplayskip=0pt\abovedisplayshortskip=0pt~\vspace*{-\baselineskip}}

\newtheoremstyle{plain}{3mm}{3mm}{\slshape}{}{\bfseries}{.}{.5em}{}
\newtheoremstyle{definition}{2mm}{2mm}{}{}{\bfseries}{.}{.5em}{}
\theoremstyle{plain}
	
\newtheorem{theorem}{Theorem}

\newtheorem{lemma}[theorem]{Lemma}
\newtheorem{proposition}[theorem]{Proposition}

\newtheorem{question}[theorem]{Question}
\theoremstyle{definition}
\newtheorem{definition}[theorem]{Definition}
\newtheorem{remark}[theorem]{Remark}
\newtheorem{example}[theorem]{Example}

\theoremstyle{plain}
\newcounter{MainTheoremCounter}

\newtheorem{Maintheorem}[MainTheoremCounter]{Theorem}

\theoremstyle{plain}
\newtheorem*{namedthm}{\namedthmname}
\newcounter{namedthm}
	

\usepackage{chngcntr}
\counterwithin{theorem}{section}

\numberwithin{equation}{section}

\allowdisplaybreaks

\usepackage{xcolor}
\definecolor{Color2}{rgb}{0.78, 0.11, 0.0}
\hypersetup{citecolor = black,colorlinks,
			linkcolor = black,
			urlcolor = Color2}

\usepackage{titlesec}
\titleformat{\section}
  {\Large\center\bfseries}
  {\thesection.}{.7em}{}
\titlespacing*{\section}{0pt}{3.5ex plus 0ex minus 0ex}{1.5ex plus 0ex}
\titleformat{\subsection}
  {\large\center\bfseries}
  {\thesubsection.}{.7em}{}
\titlespacing*{\subsection}{0pt}{3.5ex plus 0ex minus 0ex}{1.5ex plus 0ex}
\titleformat{\subsubsection}
  {\center\bfseries}
  {\thesubsubsection.}{.7em}{}
\titlespacing*{\subsubsection}{0pt}{3.5ex plus 0ex minus 0ex}{1.5ex plus 0ex}

\addto\captionsenglish{}

\usepackage{titling}
\usepackage{authblk}

\newcommand{\Cech}{\v{C}ech}

\newcommand{\supp}{{\normalfont\text{supp}}\,}

\newcommand{\eps}{\epsilon}
\newcommand{\N}{\mathbb{N}}
\newcommand{\Z}{\mathbb{Z}}

\newcommand{\C}{\mathbb{C}}
\newcommand{\Q}{\mathbb{Q}}

\newcommand{\defeq}{\vcentcolon=}

\newcommand\restr[2]{{ \left.\kern-\nulldelimiterspace #1 \right|_{#2}}}

\renewcommand{\epsilon}{\varepsilon}
\renewcommand{\leq}{\leqslant}
\renewcommand{\geq}{\geqslant}
\renewcommand{\setminus}{\backslash}

\usepackage[normalem]{ulem}

\DeclareMathOperator{\upclose}{\uparrow}

\newcommand{\syndetic}{\mathcal{S}}
\newcommand{\thick}{\mathcal{T}}

\newcommand{\dsyndetic}{d\mathcal{S}}
\newcommand{\dthick}{d\mathcal{T}}

\newcommand{\dcsyndetic}{dc\mathcal{S}}
\newcommand{\dcthick}{dc\mathcal{T}}

\newcommand{\PS}{\mathcal{PS}}
\newcommand{\dPS}{d\mathcal{PS}}
\newcommand{\dcPS}{dc\mathcal{PS}}

\newcommand{\central}{\mathcal{C}}
\newcommand{\IP}{\mathcal{IP}}

\newcommand{\filter}{\mathscr{F}}
\newcommand{\filtertwo}{\mathscr{G}}

\newcommand{\class}{\mathscr{F}}

\newcommand{\classtwo}{\mathscr{G}}
\newcommand{\classcap}{\sqcap}

\newcommand{\family}{\mathscr{F}}
\newcommand{\familyone}{\mathscr{F}}
\newcommand{\familytwo}{\mathscr{G}}
\newcommand{\familythree}{\mathscr{H}}
\newcommand{\familycap}{\sqcap}

\newcommand{\collection}{\mathscr{C}}

\newcommand{\polyret}{R}

\newcommand{\FS}{\text{FS}}

\begin{document}

\title{On sets of pointwise recurrence and dynamically thick sets}

\author[1]{Daniel Glasscock}
\author[2]{Anh N. Le}

\affil[1]{\small Dept. of Mathematics and Statistics, U. of Massachusetts Lowell, Lowell, MA, USA}
\affil[2]{\small Dept. of Mathematics,
	U. of Denver, Denver, CO, USA}

\date{}

\maketitle


\blfootnote{2020 \emph{Mathematics Subject Classification.}  Primary: 37B20. Secondary: 37B05.}


\blfootnote{\emph{Key words and phrases. }
minimal topological dynamical systems,
sets of pointwise recurrence,
dynamically thick sets,
topological recurrence,
measurable recurrence,
idempotent filters}

\begin{abstract}
A set $A \subseteq \N$ is \emph{set of pointwise recurrence} if for all minimal dynamical systems $(X, T)$, all $x \in X$, and all open neighborhoods $U \subseteq X$ of $x$, there exists a time $n \in A$ such that $T^n x \in U$.
The set $A$ is \emph{dynamically thick} if the same holds for all non-empty, open sets $U \subseteq X$.
Our main results give combinatorial characterizations of sets of pointwise recurrence and dynamically thick sets that allow us to answer questions of Host, Kra, Maass and  Glasner, Tsankov, Weiss, and Zucker.
We also introduce and study a local version of dynamical thickness called \emph{dynamical piecewise syndeticity}.  We show that dynamically piecewise syndetic sets are piecewise syndetic, generalizing results of Dong, Glasner, Huang, Shao, Weiss, and Ye.
The proofs involve the algebra of families of large sets, dynamics on the space of ultrafilters, and our recent characterization of dynamically syndetic sets. 
\end{abstract}

\setcounter{tocdepth}{2}
\tableofcontents

\vspace{0.5cm}

\section{Introduction}

\subsection{Characterizing sets of pointwise recurrence and dynamically thick sets}
\label{sec_intro_part_two}

A subset $A$ of the positive integers $\N$ is \emph{a set of (minimal, topological) pointwise recurrence} if for all minimal topological dynamical systems $(X,T)$, all $x \in X$, and all open neighborhoods $U \subseteq X$ of $x$, there exists $n \in A$ for which $T^n x \in U$.  (For the definition of minimal systems, see \cref{sec_top_dynamics}.) The set $A$ is \emph{dynamically thick} if the same holds for all non-empty, open sets $U \subseteq X$, not just the neighborhoods of $x$.
The collections of sets of pointwise recurrence and dynamically thick sets are denoted $\dcthick$ and $\dthick$, respectively, where the letter ``$c$'' in $\dcthick$ stands for ``centered'' and reflects the fact that $x \in U$ in the definition.
Sets of pointwise recurrence and dynamically thick sets are the main objects of study in this paper.

\subsubsection{Motivation: set and point recurrence in dynamics}
\label{sec_intro_set_and_pt_recurrence}

The phenomenon of recurrence in dynamical systems has been studied extensively in numerous forms.  We will discuss two forms most relevant to this work: set and point recurrence in ergodic theory and topological dynamics.

In a probability space $(X, \mu)$ with a measure-preserving transformation $T: X \to X$, a set $E \subseteq X$ of positive measure is said to recur at time $n \in \N$ if $\mu(E \cap T^{-n} E) > 0$.  A set $A \subseteq \N$ is called a \emph{set of measurable recurrence} if in every measure preserving system, every set of positive measure recurs at some time in $A$.  The question of which times a set of positive measure can be guaranteed to recur -- that is, the question of which sets are sets of measurable recurrence -- is important both conceptually and for applications.  
Szemer\'edi's theorem \cite{Szemeredi-on_the_sets}, a cornerstone of modern additive combinatorics, states that a set of integers of positive upper density contains arbitrarily long arithmetic progressions. Furstenberg \cite{Furstenberg-ErgodicBehavior} famously reproved this theorem by connecting the existence of arithmetic progressions in sets of positive upper density to the existence of sets of multiple measurable recurrence.
Since then, this connection has been explored and exploited to great effect \cite{ward_ergodic_theory_interactions_survey}.

Set recurrence in topological dynamics is also well-studied.
A set $A \subseteq \N$ is a \emph{set of topological recurrence} if for all minimal systems $(X,T)$ and all nonempty, open sets $U \subseteq X$, there exists $n \in A$ such that the set $U \cap T^{-n}U$ is nonempty.  The connection between topological dynamics and partition combinatorics forged by Furstenberg and Weiss \cite{Furstenberg_Weiss-topological-dynamics} parallels the connection between measurable dynamics and density combinatorics described above and has similarly been influential in additive combinatorics and Ramsey Theory.

Questions regarding measurable and topological set recurrence can usually be formulated in terms of the recurrence behavior of generic points: points in a set of full measure or a dense $G_{\delta}$-set return to neighborhoods of themselves at times in some set $A$.
To require, however, recurrence at times in $A$ for \emph{all} points makes no sense in the measurable setting but leads to genuinely new and interesting phenomena in the topological setting.  Despite the fact that the definitions of sets of pointwise recurrence and dynamically thick sets are naturally analogous to those in the topic of set recurrence, far less has been written on them.  Many of our results in this paper are best understood as showing that the nature of pointwise recurrence differs substantially from that of set recurrence.

\subsubsection{Sets of pointwise recurrence in the literature}

It is a consequence of a result of Pavlov \cite{Pavlov-2008} that a set of pointwise recurrence must have positive upper Banach density.  Indeed, he showed that if a set $A \subseteq \N$ has zero upper Banach density ($\limsup_{N \to \infty} \max_{n \in \N} |A \cap \{n, \ldots, n + N\}| / N = 0$) then the set $A$ is not a set of pointwise recurrence: there is a totally minimal, totally uniquely ergodic, and topologically mixing system $(X, T)$ and uncountably many points $x \in X$ for which $x \not\in \overline{\{T^n x \ \big| \ n \in A\}}$.  That dynamically thick sets and sets of pointwise recurrence are piecewise syndetic (see \cref{sec_syndetic_thick_ps_sets} for the definition) are the results of Huang and Ye \cite[Thm. 2.4]{huang_ye_2005} and Dong, Song, Ye \cite[Prop. 4.4]{dong_shao_ye_2012}, respectively.  We improve on both of these results in \cref{thm:relations_C_dPS_PS-intro}, discussed below.

Pointwise recurrence is implicitly studied in \cite{Glasner-Weiss-Interpolation,Koutsogiannis_Le_Moreira_Pavlov_Richter_interpolation_sets} in the context of interpolation sets for minimal systems. A set $I \subseteq \N \cup \{0\}$ is \emph{interpolation set for minimal systems} if for every bounded function $f: I \to \C$, there exists a minimal system $(X, T)$, a point $x \in X$, and a continuous function $F: X \to \C$ such that $f(n) = F(T^n x)$ for all $n \in I$. It was proved in \cite{Glasner-Weiss-Interpolation, Koutsogiannis_Le_Moreira_Pavlov_Richter_interpolation_sets} that $I \subseteq \N \cup \{0\}$ is an interpolation set for minimal systems if and only if $I$ is not piecewise syndetic. If $A \subseteq \N$ is not piecewise syndetic, then $I \defeq A \cup \{0\}$ is not piecewise syndetic, so there is minimal system $(X, T)$, a point $x \in X$, and a continuous function $F: X \to \C$ such that 
\[
    F(T^n x) = \begin{cases}
            1 \text{ if } n = 0 \\
            0 \text{ if } n \in A 
    \end{cases}.
\]
It follows that $A$ is not a set of pointwise recurrence. Contrapositively, a set of pointwise recurrence must be piecewise syndetic, recovering the results in \cite[Prop. 4.4]{dong_shao_ye_2012} and, hence, in \cite[Thm. 2.4]{huang_ye_2005} as well.
    
As far as we know, the term ``set of pointwise recurrence'' was first introduced by Host, Kra, and Maass \cite{HostKraMaass2016}, who compared and contrasted several different types of recurrence: single and multiple topological recurrence, topological recurrence for nilsystems, and pointwise recurrence. They observed that the family of sets of pointwise recurrence is properly contained in the family of sets of topological recurrence. In our recent work \cite{glasscock_le_2024}, we addressed several of the questions they raised. In this paper, we continue to explore the subject and answer some of their other questions.

\subsubsection{Main results: combinatorial characterizations}

For the following discussion, the reader is referred to the definitions of syndetic and thick sets in \cref{sec_syndetic_thick_ps_sets}.

A well-known combinatorial characterization for sets of topological recurrence (cf. \cite[Thm. 2.4]{bergelson_mccutcheon_1998}) reads: \emph{a set $A \subseteq \N$ is a set of topological recurrence if and only if for all syndetic sets $S \subseteq \N$, the set $A \cap (S - S)$ is nonempty}.  It is a short exercise to upgrade this to: \emph{a set $A \subseteq \N$ is a set of topological recurrence if and only if for all syndetic sets $S' \subseteq S \subseteq \N$,}
\begin{align}
\label{eqn_comb_char_top_rec}
    A \cap (S - S') \neq \emptyset.
\end{align}
In this context, Host, Kra, and Maass \cite[Question 2.13]{HostKraMaass2016} asked whether there is a purely combinatorial characterization of sets of pointwise recurrence. 

\begin{question}[{\cite[Question 2.13]{HostKraMaass2016}}]
\label{ques:characterize_sets_of_pointwise_recurrence}
Is there a combinatorial analog of pointwise recurrence? Are there sufficient conditions for being a set of pointwise recurrence?    
\end{question}

Our first result, proved in \cref{sec_dthick_comb_characterizations}, gives an answer to this question.

\begin{Maintheorem}
\label{mainthm_dct_combo_characterizations}
    Let $A \subseteq \N$.  The following are equivalent.
    \begin{enumerate}
        \item
        \label{item_intro_dct_def_char}
        The set $A$ is a set of pointwise recurrence.

        \item \label{item_intro_condition_one_dct}
        For all syndetic $S \subseteq \N$, there exists a finite set $F \subseteq A$ such that for all syndetic $S' \subseteq S$,
        \[
            F \cap (S - S') \neq \emptyset.
        \]

        \item
        \label{item_intro_dct_sif_char}
        For all $B \supseteq A$, there exists a finite set $F \subseteq \N \setminus B$ such that the set $B \cup \big( B-F \big)$ is thick.
    \end{enumerate}
\end{Maintheorem}

The characterization in \eqref{item_intro_condition_one_dct} can immediately be understood as a quantitative strengthening of the combinatorial characterization of sets of topological recurrence in \eqref{eqn_comb_char_top_rec}.  While we are not aware of any precedent for the characterization in \eqref{item_intro_dct_sif_char}, we demonstrate below in the analogous context of dynamically thick sets how it can be used to verify that a set is a set of pointwise recurrence.\\

Thick sets are easily seen to be dynamically thick. The converse is false: for all $k \in \N \setminus \{1\}$ and pairwise disjoint thick sets $H_0, \ldots, H_{k-1} \subseteq \N$, the set
\begin{align}
\label{eq:example_dynthick_1}
    \bigcup_{i=0}^{k-1} \big((k \N + i) \cap H_i \big)
\end{align}
is not thick; it is, however, dynamically thick, as we will demonstrate in \cref{ex:type2-special}.
In \cite{glasner_tsankov_weiss_zucker_2021}, Glasner, Tsankov, Weiss and Zucker call for an explicit characterization of dynamically thick sets (they call them \emph{dense orbit sets}) in countable, discrete groups.

\begin{question}[{\cite[Question 9.6]{glasner_tsankov_weiss_zucker_2021}}]
\label{ques:characterize_dense_orbit_sets}
Characterize the dynamically thick sets in countable, discrete groups.
\end{question}

Our next theorem -- proven in \cref{sec_dthick_comb_characterizations} -- gives an answer to \cref{ques:characterize_dense_orbit_sets} in $\N$ by providing combinatorial characterizations of dynamically thick sets.  We give another, ``structure-theoretic'' answer in \cref{mainthm_structure_for_dy_thick_sets} below.

\begin{Maintheorem}
\label{mainthm_dt_characterizations}
    Let $A \subseteq \N$.  The following are equivalent.
    \begin{enumerate}
        \item 
        \label{item_intro_dthick_def}
        The set $A$ is dynamically thick.

        \item \label{item:intro_complement_thick_dt_characterization_v2}
        For all piecewise syndetic $S \subseteq \N$, there exists a  finite set $F \subseteq A$ such that $S-F$ is thick.

        \item \label{item:intro_dyn_thick_idempotent}
        For all $\N \supsetneq B \supseteq A$, there exists a finite set $F \subseteq \N \setminus B$ such that the set $B-F$ is thick.
    \end{enumerate}
\end{Maintheorem}

While characterizations \eqref{item:intro_complement_thick_dt_characterization_v2} and \eqref{item:intro_dyn_thick_idempotent} given in Theorems \ref{mainthm_dct_combo_characterizations} and \ref{mainthm_dt_characterizations}  are both combinatorial, the proofs of their equivalence to \eqref{item_intro_dthick_def} involve entirely different sets of tools. That \eqref{item_intro_dthick_def} and \eqref{item:intro_complement_thick_dt_characterization_v2} are equivalent proceeds through dynamics on the space of ultrafilters.  That \eqref{item_intro_dthick_def} and \eqref{item:intro_dyn_thick_idempotent} are equivalent, on the other hand, is derived from our recent combinatorial characterizations of dynamically syndetic sets from \cite{glasscock_le_2024}.

\subsection{More on dynamically thick sets}

\subsubsection{Novel examples}
\label{sec:dyn_thick_not_ip_intro}

Theorems \ref{mainthm_dct_combo_characterizations} and \ref{mainthm_dt_characterizations} can be used to produce novel examples of sets of pointwise recurrence and dynamically thick sets.  We show in \cref{rem:combinatorial_primes}, for example, that when $(p_i)_{i = 1}^\infty$ is an enumeration of distinct primes and $(H_i)_{i = 1}^\infty$ is a sequence of thick sets, the set
\begin{align}
\label{eq:example_dct_first}
    A \defeq \bigcup_{i=1}^{\infty} \big((p_i \N + 1) \cap H_i \big)
\end{align}
is dynamically thick.
This example allows us to give a strong negative answer to the following question of Host, Kra, and Maass. For context, note that it is a consequence of a result of Furstenberg \cite{furstenberg_book_1981} that IP sets are sets of pointwise recurrence for distal systems. Host, Kra, and Maass asks whether the converse is true. (See Sections \ref{sec_top_dynamics} and \ref{sec_ip_and_central_sets} for the definitions of distal systems and IP sets, respectively.)

\begin{question}[{\cite[Question 3.11]{HostKraMaass2016}}]
\label{ques:pointwise_recurrence_not_IP}
Is it true that every set of pointwise recurrence for distal systems is an IP set? 
\end{question}

A negative answer to \cref{ques:pointwise_recurrence_not_IP} was first given in \cite{Koutsogiannis_Le_Moreira_Pavlov_Richter_interpolation_sets}, where an example of a set of pointwise recurrence for distal systems that does not contain the configuration $\{x,y,x+y\}$ was constructed.  When the thick sets $H_i$ in \eqref{eq:example_dct_first} are chosen to be sufficiently spaced, the set $A$ is dynamically thick (and, hence, a set of pointwise recurrence for all systems) but is not an IP set.  This is made precise and proven in \cref{rem:combinatorial_primes}, yielding an even stronger negative answer to \cref{ques:pointwise_recurrence_not_IP} than was previously known.

\begin{Maintheorem}
\label{Mainthm:dyn_thick_not_IP}
There exists a dynamically thick set which is not an IP set.
\end{Maintheorem}

Heuristically, the set in \eqref{eq:example_dct_first} is dynamically thick because it contains sufficiently long pieces of return-time sets from a sufficiently rich collection of systems, in this case, rotations on prime many points.
We make this heuristic precise and generalize \eqref{eq:example_dct_first} in \cref{thm_bessel} by showing that no minimal topological dynamical system can be disjoint from an infinite collection of disjoint systems.  We expect this result will be of interest outside the scope of this work.

\subsubsection{The structure of dynamical thick sets}
\label{sec_intro_disjointness}

From the definition, it is easy to see that a set is dynamically thick if and only if it has non-empty intersection with every set of the form 
\[
    R(x, U) = \{n \in \N \ | \  T^n x \in U\},
\]
where $(X, T)$ is a minimal system, $x \in X$, and $U \subseteq X$ is nonempty and open.
A set that contains a set of the form $R(x, U)$ is called \emph{dynamically syndetic}.  If, further, $x \in U$, then the set is \emph{dynamically central syndetic}.

All the examples of dynamically thick sets we have presented so far share a common form: they are comprised of intersections of thick sets and syndetic sets, where the collection of syndetic sets is sufficiently robust as to ``correlate'' with all dynamically syndetic sets.
Our next result shows that this is, in fact, the general form of all dynamically thick sets.  A collection $\collection$ of subsets of $\N$ is \emph{robustly syndetic} if for all dynamically syndetic sets $A \subseteq \N$, there exists $B \in \collection$ such that the set $A \cap B$ is syndetic.

\begin{Maintheorem}
\label{mainthm_structure_for_dy_thick_sets}
\label{cor_weak_structure_result}
    A set $A \subseteq \N$ is dynamically thick if and only if there exists a robustly syndetic collection $\collection$ of subsets of $\N$ and, for each $B \in \collection$, a thick set $H_B \subseteq \N$ such that
\begin{align}
\label{eqn_A_contains_family_cap_thick}
    A = \bigcup_{B \in \collection} \big( B \cap H_B\big).
\end{align}
\end{Maintheorem}

The dynamically thick sets  in \eqref{eq:example_dynthick_1} and \eqref{eq:example_dct_first}  correspond to the robustly syndetic collections $\{k \N + i \ | \ 0 \leq i \leq k - 1\}$ and $\{p_i \N + 1 \ | \ i \in \N\}$, respectively. More elaborate examples of robustly syndetic collections and dynamically thick sets are presented in \cref{sec:structure_dyn_thick}.  \cref{cor_weak_structure_result} provides another answer to \cref{ques:characterize_dense_orbit_sets}, one that focuses on the underlying structure of the sets rather than their combinatorial characterizations. 

\subsubsection{The splitting problem and \texorpdfstring{$\sigma$}{sigma}-compactness}
\label{sec:sigma-compact-intro}

A family $\family$ of subsets of $\N$ has the \emph{splitting property} if for all $A \in \family$, there exists a disjoint union $A = A_1 \cup A_2$ such that $A_1, A_2 \in \family$. 
It is well known to experts in the field that each of the following families has the splitting property:
    \begin{enumerate}
        \item \label{item:top_rec_list} sets of topological recurrence;
        
        \item \label{item:meas_rec_list} sets of measurable recurrence;
        
        \item \label{item:bohr_rec} sets of Bohr recurrence; and
        
        \item \label{item:nil_rec} sets of pointwise recurrence for $k$-step nilsystems.
    \end{enumerate}
Definitions of the first two families were given in \cref{sec_intro_set_and_pt_recurrence} and the last two in \cref{sec:sigma-compact}.
In this context, it is natural to ask whether the families of sets of pointwise recurrence and dynamical thick sets has the splitting property.

\begin{question}
\label{quest_sets_of_pw_rec_pproperty}
    If $A \subseteq \N$ is a set of pointwise recurrence, does there exist a disjoint partition $A = A_1 \cup A_2$ in which both $A_1$ and $A_2$ are sets of pointwise recurrence?  That is, does the family $\dcthick$ have the splitting property?  Relatedly, does the family $\dthick$ have the splitting property?
\end{question}

The established technique for proving the splitting property -- in, eg., \cite{Boshernitzan-Glasner-tworecurrence, Forrest_recurrence-in-dynamical-systems, Le-interpolation-first,Le_interpolation_for_nilsequences} -- goes as follows.  Let $N_1 < N_2 < \cdots$ be a rapidly increasing sequence of positive integers.  Assign the set $A \cap [N_1]$ to $A_1$, the set $A \cap ([N_2] \setminus [N_1])$ to $A_2$, the set $A \cap ([N_3] \setminus [N_2])$ to $A_1$, and so on.  If the family $\family$ is ``suitable'' and the $N_i$'s are chosen appropriately, the sets $A_1$ and $A_2$ will both belong to $\family$. 
Sufficient and necessary conditions for ``suitability'' of a family in the argument above seem never to have been made explicit.

In the process of attempting to answer \cref{quest_sets_of_pw_rec_pproperty}, we discovered that it is $\sigma$-compactness of dual families that underlies the splitting property for (i) -- (iv) above.  The \emph{dual} of a family $\family$ is the family of sets that have non-empty intersection with every member of $\family$.
A family is \emph{$\sigma$-compact} if the set of indicator functions of its members is a $\sigma$-compact set in $\{0,1\}^{\N}$. (See Sections \ref{sec_dual_families} and \ref{sec:compact-sigma-compact-first} for the precise definitions.)
We show in \cref{prop:sigma-compact-implies-partition-to-two} that if the dual family $\family^*$ is $\sigma$-compact, then $\family$ possesses the splitting property.
We proceed by showing that while the duals of the families \eqref{item:top_rec_list} -- \eqref{item:nil_rec} above are $\sigma$-compact, surprisingly, the duals of the families $\dthick$ and $\dcthick$ are not.

\begin{Maintheorem}
\label{theorem_dy_syndetic_is_not_sigma_compact}
    The duals of the families of dynamically thick sets and sets of pointwise recurrence are not $\sigma$-compact.
\end{Maintheorem}

\cref{theorem_dy_syndetic_is_not_sigma_compact} demonstrates why previous methods -- which in all known cases give $\sigma$-compactness -- fail at splitting sets of pointwise recurrence and dynamically thick sets into two sets in the same family.  
It also indicates that the families $\dcthick$ and $\dthick$ are more complex than others usually considered in the subject.

\subsection{Dynamically piecewise syndetic sets}
\label{sec_dy_pws_sets}

In this part of the paper, we introduce and study the family of sets formed by intersections of sets of pointwise recurrence (resp. dynamically thick sets) with members of its dual, dynamically syndetic sets (resp. dynamically central syndetic sets).
For subsets $B$ and $C$ of $\N$, the intersection $B \cap C$ is called 
\begin{itemize}
    \item \emph{dynamically piecewise syndetic} if $B$ is dynamically syndetic and $C$ is dynamically thick, and

    \item \emph{dynamically central piecewise syndetic} if $B$ is dynamically central syndetic and $C$ is a set of pointwise recurrence.
\end{itemize}
The families of dynamically piecewise syndetic sets and dynamically central piecewise syndetic sets are denoted by $\dPS$ and $\dcPS$ respectively.
To our knowledge, these families have not appeared before in the literature.

To motivate the study of these families, consider their combinatorial counterparts.  Intersections of syndetic sets and thick sets are called \emph{piecewise syndetic} sets.
Piecewise syndeticity is a ``local'' analogue of syndeticity that has the advantage of being partition regular: if a piecewise syndetic set is partitioned into finitely many subsets, one of the subsets must be piecewise syndetic.
Central sets are combinatorially-rich, dynamically-defined subsets of $\N$ that were introduced by Furstenberg \cite[Def. 8.3]{furstenberg_book_1981} to aid in the topological dynamical approach to problems in Ramsey theory. Central sets are piecewise syndetic and also enjoy partition regularity. 
The families of piecewise syndetic and central sets are important to the connection between combinatorics and dynamics (via correspondence principles) and to describing algebraic structure in the Stone-\Cech{} compactification of $\N$.

\subsubsection{Relations with other established families}

Our first results concern the relationships between the families of dynamically (central) piecewise syndetic sets, piecewise syndetic sets, and central sets.
There are dynamically syndetic sets that are not dynamically central syndetic, and there are sets of pointwise recurrence that are not dynamically thick.  Therefore, from the definitions, there is no apparent relationship between the families of dynamically piecewise syndetic and dynamically central piecewise syndetic sets.  As our terminology suggests, however, we prove that dynamically central piecewise syndetic sets are dynamically piecewise syndetic.

We could not find a proof of this fact that does not also elucidate the relationship between the families of dynamically piecewise syndetic and piecewise syndetic sets.  We show that, perhaps surprisingly, these families are the same.  The next theorem summarizes these relationships and is proved in \cref{sec_dps_sets}. Looking ahead, \cref{fig:containmentdiagram} at the end of \cref{sec_prelims} summarizes the relationships between the families studied in this paper with the established families of subsets of $\N$ in the topic.

\begin{Maintheorem}
\label{thm:relations_C_dPS_PS-intro}
    Denote by $\central$, $\PS$, $\dPS$, and $\dcPS$ the families of central, piecewise syndetic, dynamically piecewise syndetic, and dynamically central piecewise syndetic sets, respectively.  The following inclusions hold:
    \[
        \central \subsetneq \dcPS \subsetneq \dPS = \PS.
    \]
    In particular, a set is dynamically piecewise syndetic if and only if it is piecewise syndetic.
\end{Maintheorem}

That dynamically piecewise syndetic sets are piecewise syndetic offers a striking contrast to the analogy with set recurrence given in \cref{sec_intro_set_and_pt_recurrence}.  Indeed, sets of measurable recurrence -- and thus, in particular, the intersection of a set of measurable recurrence with the set of times that a particular set of positive measure recurs -- are not necessarily even of positive Banach density (and, hence, are not piecewise syndetic). For example, by a result of Furstenberg \cite{furstenberg_book_1981} and S\'{a}rk\"{o}zy \cite{sarkozy_1978}, the set $\{n^2 \ | \ n \in \N\}$ is a set of measurable recurrence.

We discussed in \cref{sec_intro_set_and_pt_recurrence} how the results in \cite[Prop. 4.4]{dong_shao_ye_2012}, \cite{Glasner-Weiss-Interpolation}, and \cite[Thm. 2.4]{huang_ye_2005} show that the families of dynamically thick sets and sets of pointwise recurrence are piecewise syndetic.  \cref{thm:relations_C_dPS_PS-intro} generalizes these results by showing that the families of dynamically (central) piecewise syndetic sets are piecewise syndetic.

\subsubsection{Applications of \cref{thm:relations_C_dPS_PS-intro}}

It is an exercise left to the interested reader to show that syndeticity and thickness are stable under changes on a non-piecewise syndetic set.  More specifically, if the symmetric difference of two subsets of $\N$ is not piecewise syndetic, then one set is syndetic/thick if and only if the other set is syndetic/thick, respectively.  Our first application of \cref{thm:relations_C_dPS_PS-intro} shows that the same is true for the dynamical counterparts of these families.

\begin{Maintheorem}\label{thm:remove_nonPS_from_dynthick}
Let $A, B \subseteq \N$. If the symmetric difference $A \triangle B$ is not piecewise syndetic, then $A$ is a set of pointwise recurrence if and only if $B$ is a set of pointwise recurrence. The same is true with ``set of pointwise recurrence'' replaced by ``dynamically thick,'' ``dynamically syndetic,'' and ``dynamically central syndetic.''
\end{Maintheorem}

This theorem is part of a more general result, \cref{cor:largest_family_satisfies_cap} below, which shows that, in fact, some of the properties in \cref{thm:remove_nonPS_from_dynthick} characterize non-piecewise syndeticity.\\

As another application of \cref{thm:relations_C_dPS_PS-intro}, we demonstrate how we can upgrade some of the results in our companion paper \cite{glasscock_le_2024} that concern sets of pointwise recurrence.

A set $A \subseteq \N$ is called a \emph{set of multiple topological recurrence} if for all $k \in \N$, all minimal systems $(X, T)$, and all nonempty, open sets $U \subseteq X$, there exists $n \in A$ such that
\[
    U \cap T^{-n} U \cap \cdots \cap T^{-kn} U \neq \emptyset.
\]
Just as with sets of (single) topological recurrence (defined in \cref{sec_intro_set_and_pt_recurrence}), it is important both conceptually and for applications to understand which sets are sets of multiple topological recurrence.
While it is known that a set of (single) topological recurrence is not necessarily a set of multiple topological recurrence \cite{frantzikinakis_lesigne_wierdl_2006, furstenberg_book_1981}, Host, Kra, and Maass \cite[Question 2.14]{HostKraMaass2016} asked whether or not a set of pointwise recurrence is a set of multiple topological recurrence.

A positive answer to this question was given in our recent work \cite{glasscock_le_2024}: sets of pointwise recurrence are, in fact, sets of polynomial multiple measurable recurrence for commuting transformations; see the precise definition in \cref{sec_sets_of_poly_rec_and_brauer}.
We also showed in \cite{glasscock_le_2024} that sets of pointwise recurrence contain Brauer-type polynomial configurations, i.e. patterns of the form
\begin{align}
\label{eqn_Brauer}
    x, y, x + p_1(y), \ldots, x + p_k(y),
\end{align}
where $p_1, \ldots, p_k \in \Q[x]$, $p_i(\Z) \subseteq \Z$, and $p_i(0) = 0$.  The following theorem strengthens both of these results by showing that not only do sets of pointwise recurrence have these properties, but so do the intersections of sets of pointwise recurrence with dynamically central syndetic sets. 

\begin{Maintheorem}
\label{main_thm_brauer}
    Dynamically central piecewise sets are sets of polynomial multiple measurable recurrence for commuting transformations and contain the Brauer-type configurations described in \eqref{eqn_Brauer} above.
\end{Maintheorem}

The extent to which dynamically central piecewise syndetic sets are combinatorially rich remains an interesting open question that we speculate on in \cref{sec_open_quests_dps_sets}.

\subsection{A companion paper on dynamically syndetic sets}

Dual to the families of sets of pointwise recurrence and dynamically thick sets are the families of dynamically central syndetic and dynamically syndetic sets.  We study these families in depth in a companion paper \cite{glasscock_le_2024}.  Our main results in this paper rely critically on the results from that paper that characterize dynamical syndeticity in terms of membership in a syndetic, idempotent filter. On the other hand, while we upgrade some of the results from \cite{glasscock_le_2024} in this one, none of the results from the companion paper rely logically on the present one.

\subsection{Organization of the paper}

Notation, terminology, and preliminaries are given in \cref{sec_prelims}.
The main body of the work is then divided into four parts:
\begin{itemize}
     \item[--] \cref{sec_dthick}, in which we give combinatorial characterizations of sets of pointwise recurrence  and dynamically thick sets and prove Theorems \ref{mainthm_dct_combo_characterizations} and \ref{mainthm_dt_characterizations};

     \item[--] \cref{sec:structure_dyn_thick}, in which give examples of dynamically thick sets, prove a structure theorem for them, and prove Theorems \ref{Mainthm:dyn_thick_not_IP} and \ref{mainthm_structure_for_dy_thick_sets};

    \item[--] \cref{sec:sigma-compact}, in which we address the splitting problem and $\sigma$-compactness of dual families and prove \cref{theorem_dy_syndetic_is_not_sigma_compact};

    \item[--] \cref{sec_dps_sets}, in which we introduce and study dynamically (central) piecewise syndetic sets and prove Theorems \ref{thm:relations_C_dPS_PS-intro}, \ref{thm:remove_nonPS_from_dynthick}, and \ref{main_thm_brauer}.
\end{itemize}
We conclude the paper with a collection of open problems and directions in \cref{sec:open_questions}.  

\subsection{Acknowledgments}

The authors thank Neil Hindman for his correspondence regarding translates of central sets. 
We are grateful to Andy Zucker and Josh Frisch for helpful discussions at the 2024 Southeastern Logic Symposium. Thanks are also due to Eli Glasner, John Johnson, and Mauro Di Nasso for their assistance in correcting several historical references and claims made in an earlier draft. Finally, we are indebted to Andreas Koutsogiannis, Joel Moreira, Florian Richter, and Donald Robertson, whose collaboration helped sharpen a number of the ideas presented here.

\section{Preliminaries}
\label{sec_prelims}

The set of positive integers, non-negative integers, and integers are denoted by $\N$, $\N_0$, and $\Z$, respectively.  For $N \in \N_0$, we define $[N]$ to be $\{0, 1, \ldots, N - 1\}$, and $\mathcal{P}(\N)$ denotes the power set of $\N$.  In what follows, note that all set operations on subsets of $\N$ result in subsets of $\N$.

For $A, B \subseteq \N$ and $n \in \N$, we define
\begin{align*}
    A+n &\defeq \{a + n \ | \ a \in A \}, & nA & \defeq \{ na \ | \ a \in A\}, \\
    A-n &\defeq \{m \in \N \ | \ m + n \in A \}, & A/n & \defeq \{ m \in \N \ | \ mn \in A\},\\
    A+B &\defeq \bigcup_{b \in B} (A + b), & A-B &\defeq \bigcup_{b \in B} (A-b).
\end{align*}
In this paper, the empty union of subsets of $\N$ is empty while the empty intersection of subsets of $\N$ is equal to $\N$.

\subsection{Families of positive integers}
\label{sec_combinatorics}

Much of the algebra in this subsection originates in the work of Choquet \cite{choquet_1947} and Schmidt \cite{schmidt_1952,schmidt_1953}. A modern overview appears in \cite[Section 2]{christopherson_johnson_2022}, and generalizations to arbitrary semigroups are given in \cite{Burns_Davenport_Frankson_Griffin_Johnson_Kebe_syndetic}.

The \emph{upward closure} of a collection $\class$ of subsets of $\N$ is
\[\upclose \class \defeq \big\{ B \subseteq \N \ \big| \ \exists A \in \class, \ B \supseteq A \big\}.\]
If $\class = \upclose \class$, we say that $\class$ is \emph{upward-closed}.  We call an upward-closed collection of subsets of $\N$ a \emph{family}.  Both $\emptyset$ and $\mathcal{P}(\N)$ are families; all other families are called \emph{proper}.

\subsubsection{Dual families and intersections}
\label{sec_dual_families}

Let $\family$ be a family of subsets of $\N$.  The collection of all those sets which have nonempty intersection with all elements of $\family$, that is,
\[\family^* \defeq \big\{ B \subseteq \N \ \big | \ \forall A \in \family, \ B \cap A \neq \emptyset \big\},\]
is the \emph{dual} of $\family$.
We will have need for the following easy facts:
\begin{enumerate}
    \item $(\family^*)^* = \family$, justifying the term \emph{dual};
    \item $\emptyset$ and $\mathcal{P}(\N)$ are dual families, and the dual of a proper family is a proper family;
    \item $A \in \family$ if and only if $\N \setminus A \not\in \family^*$;
    \item $\family \subseteq \familytwo$ if and only if $\familytwo^* \subseteq \family^*$.\\
\end{enumerate}

For nonempty families $\familyone$ and $\familytwo$, we define
\[\familyone \familycap \familytwo \defeq \big\{ A \cap B \ \big| \ A \in \familyone \text{ and } B \in \familytwo \big\}.\]
We will adopt the convention that $\emptyset \classcap \family = \family \classcap \emptyset = \family$ for all families $\family$.  The following facts are quick to check for all families $\family$ and $\familytwo$:
\begin{enumerate}
\setcounter{enumi}{4}
    \item the collection $\family \classcap \familytwo$ is a family that contains both $\family$ and $\familytwo$ (despite what is suggested by the notation);
    \item if $\familyone$ is a proper family, then $\familyone \classcap \familyone^*$ is a proper family;
    \item if $\family_1 \subseteq \familytwo_1$ and $\family_2 \subseteq \familytwo_2$, then $\family_1 \classcap \family_2 \subseteq \familytwo_1 \classcap \familytwo_2$.
\end{enumerate}

\subsubsection{Filters, partition regularity, and ultrafilters}
\label{sec_filter_pr_and_ufs}

A family $\family$ is a \emph{filter} if for all $A_1, A_2 \in \family$, the set $A_1 \cap A_2 \in \family$.
A family $\family$ is \emph{partition regular} if for all $A \in \family$ and all finite partitions $A = A_1 \cup \cdots \cup A_k$, some piece $A_i$ belongs to $\family$.
These are dual notions (see fact \eqref{item_pr_filter_duality} below), both of which capture a kind of combinatorial richness imporant in this subject.

A family that has some property $P$ and that is also a filter is called a \emph{$P$ filter}. We will consider in the next subsection, for example, translation-invariant filters. Note that $\emptyset$ and $\mathcal{P}(\N)$ are both partition regular filters.

The following facts are quick to check:
\begin{enumerate}
    \item \label{item_pr_filter_duality} a family $\family$ is a filter if and only if the dual family $\family^*$ is partition regular;
    \item if $\family$ is a proper filter, then $\family \subseteq \family^*$;
    \item \label{item_filter_containment} if $\familyone$ and $\familytwo$ are filters, then $\familyone \classcap \familytwo$ is the smallest filter containing both $\familyone$ and $\familytwo$; in particular, if $\family \subseteq \familytwo$, then $\family \classcap \familytwo = \familytwo$.\\ 
\end{enumerate}

Given a family $\family$, we will frequently consider the family $\family \classcap \family^*$.  The follow lemma records some useful facts about it.

\begin{lemma}
\label{lemma_class_cap_is_pr}
    Let $\family$ be a family and define $\familytwo \defeq \family \classcap \family^*$.  The family $\familytwo$ is partition regular and its dual, $\familytwo^*$, is a filter.  Moreover, for $A \subseteq \N$, the following are all equivalent to the set $A$ belonging to $\familytwo^*$:
    \begin{enumerate}
        \item \label{item_a_cap_b_is_in_family} for all $B \in \family$, the set $A \cap B \in \family$;
        \item for all $B \in \family^*$, the set $A \cap B \in \family^*$;
        \item for all $B \in \familytwo$, the set $A \cap B \in \familytwo$;
        \item for all $B \in \familytwo^*$, the set $A \cap B \in \familytwo^*$.
    \end{enumerate}
\end{lemma}
\begin{proof}
    If $\family$ is not proper, then $\familytwo = \mathcal{P}(\N)$, which is partition regular.  Otherwise, the partition regularity of $\familytwo$ is shown in \cite[Prop. 2.5 (h)]{christopherson_johnson_2022}.  That its dual, $\familytwo^*$, is a filter is mentioned as item \eqref{item_pr_filter_duality} before the statement of this lemma.

    We will show that \eqref{item_a_cap_b_is_in_family} is equivalent to $A$ belonging $\familytwo^*$; the other statements follow analogously and are left to the reader.  Suppose $A \in \familytwo^*$ and $B \in \family$.  To see that $A \cap B \in \family$, we will show that for all $C \in \family^*$, the set $A \cap B \cap C$ is nonempty.  Let $C \in \family^*$.  The set $B \cap C \in \familytwo$, and since $A \in \familytwo^*$, we have that $A \cap B \cap C \neq \emptyset$, as desired.
    
    Conversely, suppose that \eqref{item_a_cap_b_is_in_family} holds.  To see that $A \in \familytwo^*$, we will show that for all $D \in \familytwo$, the set $A \cap D \neq \emptyset$.  Let $D \in \familytwo$.  There exists $B \in \family$ and $C \in \family^*$ such that $D = B \cap C$.  Then $A \cap D = A \cap B \cap C$.  By assumption, $A \cap B \in \family$, whereby $ A \cap B \cap C \neq \emptyset$, as desired.
\end{proof}

A proper family is called an \emph{ultrafilter} if it satisfies one (equivalently, all) of the properties in the following lemma.  This lemma is well-known (cf. \cite[Thm. 3.6]{hindman_strauss_book_2012}), so we omit the proof.

\begin{lemma}
\label{lemma_equiv_ultrafilter}
    Let $\family$ be a proper family of subsets of $\N$.  The following are equivalent:
    \begin{enumerate}
        \item
        \label{item_family_is_maximal}
        the family $\family$ is a proper filter and is maximal (by containment) amongst proper filters;
        
        \item
        \label{item_family_is_self_dual}
        the family $\family$ is contained in a proper filter and is self-dual, ie., $\family = \family^*$;
        
        \item
        \label{item_family_is_pr_and_filter}
        the family $\family$ is partition regular and a filter.
    \end{enumerate}
\end{lemma}

Ultrafilters exist by a standard application of Zorn's Lemma.  The space of ultrafilters on $\N$ is an important object that we will make use of in this paper.  We delay further discussion until \cref{sec_topology_and_ultrafilters}.

\subsubsection{The algebra of family translations}
\label{sec_family_algebra}

For families $\family$ and $\familytwo$ of subsets of $\N$, $A \subseteq \N$, and $n \in \N$, define
\begin{align*}
    A - \family &\defeq \big\{ n \in \N \ \big| \ A-n \in \family \big\}, \text{ that is, } n \in A - \family \text{ if and only if } A - n \in \family, \\
    \familyone + \familytwo &\defeq \big\{ B \subseteq \N \ \big| \ B - \familytwo \in \family\big\}, \text{ that is, $B \in \family + \familytwo$ if and only $B - \familytwo \in \family$}.
\end{align*}

The reader will quickly verify that the sum of two families is a family. This definition appears in \cite{berglund_hindman_1984,papazyan_1989} for filters and agrees with the usual definition of sums of ultrafilters (cf. \cite[Thm. 4.12]{hindman_strauss_book_2012}).
Roughly, the containment $\family \subseteq \familytwo + \familythree$ means: for all $A \in \family$, there are $\familytwo$ many positive integers $n$ for which the set $A-n$ belongs to $\familythree$.

\begin{lemma}
    \label{lemma_filter_algebra}
    \label{lemma_fam_dual_sums}
    \label{lemma_filter_sums_containment}
    Let $n \in \N$, $A, B \subseteq \N$, and $\filter, \filtertwo, \filter_1, \filter_2, \filtertwo_1, \filtertwo_2$ be families.
    \begin{enumerate}
        \item
        \label{item_family_prop_zero}
        If $\filter$ and $\filtertwo$ are proper, then $\filter + \filtertwo$ is proper.
        
        \item
        \label{item_family_prop_one_new}
        $(A - \filter) \cap (B - \filtertwo) \subseteq (A \cap B) - (\filter \classcap \filtertwo)$, with equality if $\filter = \filtertwo$ is a filter.

        \item \label{item_family_prop_two}
        If $\filter$ and $\filtertwo$ are filters, then $\filter + \filtertwo$ is a filter.

        \item \label{item_family_prop_three}
        $(A -n) - \filter = (A - \filter) - n$.

        \item \label{item_family_prop_four}
        $A - (\filter + \filtertwo) = (A - \filtertwo) - \filter$.

        \item \label{item_family_diff_containment}
        If $\family \subseteq \familytwo$, then $A - \family \subseteq A - \familytwo$.
        
        \item \label{item_filter_sums_containment}
        If $\filter_1 \subseteq \filtertwo_1$ and $\filter_2 \subseteq \filtertwo_2$, then $\filter_1 + \filter_2 \subseteq \filtertwo_1 + \filtertwo_2$.

        \item \label{item_sum_dual}
        $(\familyone + \familytwo)^* = \familyone^* + \familytwo^*$.

        \item \label{item_cap_sums}
        $(\familyone_1 + \familyone_2) \classcap (\familytwo_1 + \familytwo_2) \subseteq (\familyone_1 \classcap \familytwo_1) + (\familyone_2 \classcap \familytwo_2)$.
    \end{enumerate}
\end{lemma}

\begin{proof}
    The statements \eqref{item_family_prop_three}, \eqref{item_family_prop_four}, and \eqref{item_sum_dual} are proved in \cite[Lemma 2.1]{glasscock_le_2024}. We now prove the other statements. 

    \eqref{item_family_prop_zero} \ Suppose $\family$ and $\familytwo$ are proper.  Since $\emptyset \not\in \familytwo$, we have $\emptyset - \familytwo = \emptyset$. Since $\emptyset \not\in \family$, we have $\emptyset \not\in \family + \familytwo$. Since $\family$ and $\familytwo$ both contain $\N$, we see that $\N \in \family + \familytwo$.  Since $\family + \familytwo$ does not contain $\emptyset$ but contains $\N$, it is proper.
    
    \eqref{item_family_prop_one_new} \  Suppose $n \in (A - \filter) \cap (B - \filtertwo)$.  Since $A - n \in \filter$ and $B - n \in \filtertwo$, we have that $(A \cap B) - n = (A-n) \cap (B-n) \in \filter \classcap \filtertwo$.  Thus, $n \in (A \cap B) - (\filter \classcap \filtertwo)$, as desired.

    If $\family = \familytwo$ is a filter, then $\family \classcap \familytwo = \family$.  We wish to show in this case that $(A \cap B) - \family \subseteq (A- \family) \cap (B-\family)$.  Suppose $n \in (A \cap B) - \family$ so that $(A - n) \cap (B-n) \in \family$.  Since $\family$ is upward closed, we have that $A -n \in \family$ and $B-n \in \family$.  Therefore, $n \in (A - \family) \cap (B-\family)$, as desired.

    \eqref{item_family_prop_two} \ Suppose $\filter$ and $\filtertwo$ are filters.  Since $\filter + \filtertwo$ is a family, we need only to show that for all $A, B \in \filter + \filtertwo$, the set $A \cap B \in \filter + \filtertwo$.  Let $A, B \in \filter + \filtertwo$.  Thus, $A - \filtertwo, B - \filtertwo \in \filter$, whereby $(A-\filtertwo) \cap (B-\filtertwo) \in \filter$.  By \eqref{item_family_prop_one_new}, it follows that $(A \cap B) - \filtertwo \in \filter$, whereby $A \cap B \in \filter + \filtertwo$, as desired.

    \eqref{item_family_diff_containment} \ If $n \in A - \family$, then $A - n \in \family \subseteq \familytwo$, whereby $n \in A - \familytwo$.
    
    \eqref{item_filter_sums_containment} \ Suppose $A \in \filter_1 + \filter_2$.  Then $A-\filter_2 \in \filter_1 \subseteq \filtertwo_1$.  Since $\filter_2 \subseteq \filtertwo_2$, we have that $A - \filter_2 \subseteq A - \filtertwo_2$.  Since $\filtertwo_1$ is upward closed and $A- \filter_2 \in \filtertwo_1$, we see that $A - \filtertwo_2 \in \filtertwo_1$.  Therefore, $A \in \filtertwo_1 + \filtertwo_2$, as desired.

    \eqref{item_cap_sums} \ Suppose $\family_1 + \family_2 = \emptyset$ or $\familytwo_1 + \familytwo_2 = \emptyset$.  If $\family_1 + \family_2 = \emptyset$, then either $\family_1 = \emptyset$ or $\family_2 = \emptyset$.  In the first case, the desired inclusion simplifies to $\familytwo_1 + \familytwo_2 \subseteq \familytwo_1 + (\family_2 \classcap \familytwo_2)$, which follows from \eqref{item_filter_sums_containment}.  The second case is shown similarly.  The case that $\familytwo_1 + \familytwo_2 = \emptyset$ is handled analogously.

    Otherwise, we have that $\family_1 + \family_2 \neq \emptyset$ and $\familytwo_1 + \familytwo_2 \neq \emptyset$.  To see the desired inclusion, let $C \in (\familyone_1 + \familyone_2) \classcap (\familytwo_1 + \familytwo_2)$. There exists $A \in \familyone_1 + \familyone_2$ and $B \in \familytwo_1 + \familytwo_2$ such that $C = A \cap B$. It follows from point \eqref{item_family_prop_one_new} that $(A-\family_2) \cap (B-\familytwo_2) \subseteq (A \cap B) - (\family_2 \classcap \familytwo_2)$.  Since $A-\family_2 \in \family_1$ and $B-\familytwo_2 \in \familytwo_1$, we see that that $(A \cap B) - (\family_2 \classcap \familytwo_2) \in \family_1 \classcap \familytwo_1$, whereby $A \cap B \in (\family_1 \classcap \familytwo_1) + (\family_2 \classcap \familytwo_2)$, as desired.
\end{proof}

\subsubsection{Translation-invariance and idempotency}
\label{sec_trans_inv_and_idempotent_families}

For a family $\family$ of subsets of $\N$ and $m \in \N$, define
\[\family - m \defeq \big\{ B - m \ \big| \ B \in \family \big\}.\]
If $\family - n \subseteq \family$ for all $n \in \N$, that is, if $\family \subseteq \{\N\} + \family$, we say that $\family$ is \emph{translation invariant}.
If $A - \filter \in \filter$ for all $A \in \filter$, that is, if $\filter \subseteq \filter + \filter$, we say that $\filter$ is \emph{idempotent}.
It is immediate that translation-invariant families are idempotent.

\begin{lemma}
\label{lemma_idempotent_classcap_is_idempotent}
    Let $\family$ and $\familytwo$ be families.
    \begin{enumerate}
        \item
        \label{item_class_cap_of_idempotent_is_idempotent}
        If $\family$ and $\familytwo$ are idempotent, then $\family \classcap \familytwo$ is idempotent.

        \item
        \label{item_class_cap_of_trans_inv}
        If $\family$ and $\familytwo$ are translation invariant, then $\family \classcap \familytwo$ is translation invariant.
    \end{enumerate}
\end{lemma}

\begin{proof}
    \eqref{item_class_cap_of_idempotent_is_idempotent} \ We appeal to \cref{lemma_fam_dual_sums} \eqref{item_cap_sums} with $\familyone_1 = \familyone_2 = \familyone$ and $\familytwo_1 = \familytwo_2 = \familytwo$.  Thus,
    \[\familyone \classcap \familytwo \subseteq (\familyone + \familyone) \classcap (\familytwo + \familytwo) \subseteq (\familyone \classcap \familytwo) + (\familyone \classcap \familytwo),\]
    as desired.  Statement \eqref{item_class_cap_of_trans_inv} follows similarly and is left to the reader.
\end{proof}

Idempotent ultrafilters are important objects with a deep history in the subject (cf. \cite{bergelson_2010}).  It appears that there are two different definitions of ``idempotent ultrafilter.''  In the literature, a family $\family$ is an idempotent ultrafilter if it is an ultrafilter and satisfies $\family = \family + \family$.  According to our definitions, an idempotent ultrafilter is an idempotent family (ie. $\family \subseteq \family + \family$) that is also an ultrafilter. We show in the following lemma these two definition are, in fact, the same.  We will use facts from \cref{sec_topology_and_ultrafilters}.

\begin{theorem}
\label{thm_equiv_forms_of_max_idempot_filter}
    Let $\family$ be a family.  The following are equivalent:
    \begin{enumerate}
        \item
        \label{item_max_prop_idemp_filter}
        the family $\family$ is a proper, idempotent filter and is maximal (by containment) in the set of proper, idempotent filters;

        \item
        \label{item_ultra_idemp}
        the family $\family$ is an ultrafilter that satisfies $\family \subseteq \family + \family$;

        \item
        \label{item_ultra_strong_idemp}
        the family $\family$ is an ultrafilter that satisfies $\family = \family + \family$.
    \end{enumerate}
    In particular, every proper, idempotent filter is contained in an idempotent ultrafilter.
\end{theorem}

\begin{proof}
    \eqref{item_max_prop_idemp_filter} $\implies$ \eqref{item_ultra_idemp} \ We will prove first that every proper, idempotent filter is contained in an idempotent family that is an ultrafilter.  Supposing this is so, if $\family$ is a maximal proper, idempotent filter, then it is contained in an idempotent family $\familytwo$ that is an ultrafilter.  By the maximality of $\family$ as a proper, idempotent filter, we have that $\family = \familytwo$, giving \eqref{item_ultra_idemp}, as desired.

    Suppose $\family$ is a proper, idempotent filter and is maximal as such. We claim that the set $\overline{\filter} \defeq \bigcap_{A \in \family} \overline{A}$ is a nonempty, compact subsemigroup of $(\beta \N, +)$.  Because ultrafilters are maximal filters, there is an ultrafilter containing $\family$, and so $\overline{\filter}$ is nonempty.   The set $\overline{\filter}$ is compact because it is closed and $\beta \N$ is compact.  Suppose $q_1, q_2 \in \overline{\filter}$.  Since $\filter \subseteq q_1$ and $\filter \subseteq q_2$, it follows from \cref{lemma_filter_sums_containment} \eqref{item_filter_sums_containment} and the idempotency of $\filter$ that
    \[\filter \subseteq \filter + \filter \subseteq q_1 + q_2.\]
    Therefore, we have that $q_1 + q_2 \in \overline{\filter}$.

    Finally, since $\overline{\filter}$ is a nonempty, compact subsemigroup of $(\beta \N, +)$, by \cite[Thm. 2.5]{hindman_strauss_book_2012}, there exists an idempotent ultrafilter containing $\family$.

    \eqref{item_ultra_idemp} $\implies$ \eqref{item_ultra_strong_idemp} \ Suppose $\family$ is an ultrafilter that satisfies $\family \subseteq \family + \family$.  By \cref{lemma_filter_algebra} \eqref{item_family_prop_zero} and \eqref{item_family_prop_two}, the family $\family + \family$ is a proper filter.  Since $\family$ is a maximal proper filter, we have that $\family = \family + \family$, as desired.

    \eqref{item_ultra_strong_idemp} $\implies$ \eqref{item_max_prop_idemp_filter} \ If $\family$ is an ultrafilter that satisfies $\family = \family + \family$, then it is a proper, idempotent filter.  Since it is maximal amongst proper filters, it is maximal amongst proper, idempotent filters.
\end{proof}

\begin{remark}
\label{rmk_idempotents_without_zorn}
    A routine application of Zorn's lemma gives the existence of maximal proper, idempotent filters.  It follows from \cref{thm_equiv_forms_of_max_idempot_filter} that these are ultrafilters.  The proof given here relies on the fact that a compact, Hausdorff, right-topological semigroup contains an idempotent element (see \cite[Thm. 2.5]{hindman_strauss_book_2012} and the end-of-chapter notes). A short, topology-free proof of this was given by Papazyan \cite{papazyan_1989}, and a proof without the full strength of Zorn's Lemma was given by Di Nasso and Tachtsis \cite{dinasso_tachtsis_2018}.

    Another routine application of Zorn's lemma gives the existence of maximal proper, translation-invariant filters.  It is an easy exercise to show that there are no translation-invariant ultrafilters on $\N$.  Therefore, in contrast to maximal proper, idempotent filters, maximal proper, translation-invariant filters are not ultrafilters.
\end{remark}

\subsection{Topology and dynamics}

Let $X$ be a topological space.  The closure and interior of $U \subseteq X$ are denoted $\overline U$ and $U^\circ$, respectively. If $X$ is a metric space, the open ball centered around $x \in X$ with radius $\eps > 0$ is denoted $B_\eps(x)$.

\subsubsection{Topological dynamics}
\label{sec_top_dynamics}

A nonempty, compact metric space $(X,d_X)$ together with a continuous self-map $T: X \to X$ is called a \emph{(topological dynamical) system}.  
A system of the form $(Y,T)$, where $Y \subseteq X$ is nonempty, closed, and \emph{$T$-invariant}, meaning $TY \subseteq Y$, is a \emph{subsystem of $(X,T)$}.
A system is \emph{minimal} if its only subsystem is itself.
Given two systems $(X,T)$ and $(Y,S)$, a continuous surjection $\pi: X \to Y$ that satisfies $S \circ \pi = \pi \circ T$ is called a \emph{factor map} of systems.  We write $\pi: (X,T) \to (Y,S)$ to indicate that $\pi$ is a factor map.

All systems in this paper will be considered as actions of the semigroup $(\N, +)$ by continuous maps on a compact metric space.  Thus, even in the event that $T$ is a homeomorphism (in which case, the system $(X,T)$ is called \emph{invertible}), the emphasis will be on non-negative iterates of the map $T$.

We will occasionally have need to consider non-metrizable systems.  A \emph{not-necessarily-metrizable system} is a nonempty, compact Hausdorff space $X$ together with a continuous self-map $T: X \to X$.  We will always specify when a system need not be metrizable.  All terminology and notation will be shared without confusion between the metric and non-metric settings.

For a system $(X,T)$, a point $x \in X$, and a set $A \subseteq \N$, define
\[T^A x \defeq \big\{ T^a x \ \big| \ a \in A \big\}.\]
The set $T^{\N} x$ is called \emph{the orbit of $x$} and its closure, $\overline{T^{\N}x}$, is \emph{the orbit closure of $x$}. Note that $(\overline{T^{\N}x}, T)$ is a subsystem of $(X,T)$, whereby a system $(X,T)$ is minimal if and only if every point $x \in X$ has a dense orbit.

For a system $(X,T)$, a set $U \subseteq X$, and a point $x \in X$, write
\begin{align*}
    R(x, U) &\defeq \big\{n \in \N \ \big | \ T^n x \in U \big\}
\end{align*}
for the set of times at which the point $x$ visits $U$. The reader will quickly verify that $n \in \N_0$,
\[R(x,U) - n = R(T^n x, U) = R(x, T^{-n}U).\]

Syndetic sets -- defined in \cref{sec_syndetic_thick_ps_sets} -- are closely related to minimal dynamics.  We will make use of two facts that can be quickly verified using the results in \cite[Ch. 1, Sec. 4]{furstenberg_book_1981}.  Let $(X,T)$ be a not-necessarily-metrizable system.
\begin{enumerate}
    \item Let $x \in X$.  The subsystem $(\overline{T^{\N_0}x}, T)$ is minimal if and only if for all open $U \subseteq X$ containing $x$, the set $R(x,U)$ is syndetic. Such a point $x$ is called \emph{uniformly recurrent.}

    \item The system $(X,T)$ is minimal if and only if for all nonempty, open $U \subseteq X$ and all $x \in X$, the set $R(x,U)$ is syndetic.
\end{enumerate}

The dual notions of distality and proximality arise a few times in this paper.  Let $(X,T)$ be a system and $x, y \in X$.  The points $x$ and $y$ are \emph{proximal} if $\inf_{n \in \N} d_X(T^n x, T^n y) = 0$.  The point $x$ is called \emph{distal} if it is proximal only to itself.  The system $(X,T)$ is called \emph{distal} if every point is distal.

\subsubsection{Symbolic dynamics}
\label{sec:symbolic_prelim}

We denote by $\{0,1\}^{\N}$ the set of 0--1 valued functions on $\N$. This is a compact Hausdorff space under the product topology.
Instead of specifying one of the many metrics that generate this topology, let us note only that two functions in $\{0,1\}^{\N}$ are near if and only if they agree when restricted to a long initial interval $\{1, \ldots, N\}$ of $\N$.

The value of the function $\omega \in \{0,1\}^{\N}$ at $i \in \N$ is denoted by $\omega(i)$.
The \emph{full (left) shift} is the system $(\{0,1\}^{\N}, \sigma)$, where the (left) shift $\sigma: \{0,1\}^{\N} \to \{0,1\}^{\N}$ at $\omega$ is defined by $(\sigma \omega)(i) = \omega(i+1)$, $i \in \N$.  Starting at $0$ instead of $1$, everything above applies to the space $\{0,1\}^{\N_0}$ and the \emph{full (left) shift} $(\{0,1\}^{\N_0}, \sigma)$.

\subsubsection{Dynamics on the space of ultrafilters}
\label{sec_topology_and_ultrafilters}

The Stone-\Cech{} compactification of $\N$, denoted $\beta \N$, is a compact Hausdorff, right-topological semigroup that provides a convenient universal object in the category of minimal systems.  We summarize here just what we will need; the interested reader is directed to \cite[Ch. 19]{hindman_strauss_book_2012} and the references therein for more information.

As a set, we realize $\beta \N$ as the set of all ultrafilters on $\N$ (cf. \cref{sec_filter_pr_and_ufs}). Sets of the form
\[\overline{A} \defeq \big\{p \in \beta \N \ \big| \ A \in p \big\}, \quad A \subseteq \N,\]
form a base for a non-metrizable, compact Hausdorff topology on $\beta \N$.  Since ultrafilters are families, addition of two ultrafilters is as defined in \cref{sec_family_algebra}.  It follows by combining \cref{lemma_equiv_ultrafilter} \eqref{item_family_is_self_dual} and \cref{lemma_filter_algebra} \eqref{item_family_prop_zero}, \eqref{item_family_prop_two}, and \eqref{item_sum_dual} that the sum of two ultrafilters is an ultrafilter.  Thus, $(\beta \N, +)$ is a compact Hausdorff semigroup.

It is not hard to check that for all $n \in \N$ and $q \in \beta \N$, the maps $\beta \N \to \beta \N$ defined by $p \mapsto n + p$ and $p \mapsto p + q$ are continuous.  Thus, $(\beta \N, +)$ is called a \emph{right-topological} semigroup.  A nonempty subset $L \subseteq \beta \N$ with the property that $p + L \subseteq L$ for all $p \in \beta \N$ is called a \emph{left ideal}.  A left ideal that is minimal amongst all left ideals is called a \emph{minimal left ideal}.  By a routine application of Zorn's lemma, minimal left ideals exist in $(\beta \N, +)$.  If $L \subseteq \beta \N$ is a minimal left ideal, for all $p \in L$, the set $\beta \N + p$ is a left ideal contained in $L$, and hence is equal to $L$.  Since addition on the right by $p$ is continuous and $\beta \N$ is compact, we see that minimal left ideals are compact.

Denoting addition by 1 on the left by $\lambda_1: \beta \N \to \beta \N$, we see that $(\beta \N, \lambda_1)$ is a non-metrizable topological dynamical system.  Fix $q \in \beta \N$.  Since addition on the right by $q$ is continuous and $\N$ is dense in $\beta \N$, we have that the orbit closure of $q$ under $\lambda_1$ is
\[\overline{\lambda_1^{\N} q} = \overline{ \N + q } = \overline{\N} + q = \beta \N + q,\]
a left ideal.  It is not hard to see that $\beta \N + q$ is a minimal left ideal if and only if the system $(\beta \N + q, \lambda_1)$ is minimal. (Indeed, the system $(\beta \N + q, \lambda_1)$ is minimal if and only if for all $p \in \beta \N + q$, the orbit closure $\beta \N + p$ -- a left ideal -- is equal to $\beta \N + q$.)  In this case, the ultrafilter $q$ is called a \emph{minimal ultrafilter}.  We see, then, that the minimal left ideals of the semigroup $(\beta \N, +)$ are precisely the minimal subsystems of $(\beta \N, \lambda_1)$.

Let $(X,T)$ be a not-necessarily-metrizable system, and let $x \in X$.  Consider the map $T^{\displaystyle \cdot} x: \N \to X$ given by $n \mapsto T^n x$.  By the universal property of the Stone-\Cech{} compactification, the map $T^{\displaystyle \cdot} x$ lifts to a continuous map $T^{\displaystyle \cdot} x: \beta \N \to X$.  Concretely, for $p \in \beta \N$, the point $T^p x \in X$ is defined uniquely by the property that for all open $U \subseteq X$ containing $T^p x$, the set $R(x,U)$ is a member of $p$.  The following fact will be useful for us later on: if $U \subseteq X$ is clopen and $p \in \beta \N$, then
\begin{align}
    \label{eqn_ultrafilter_shift_of_return_set}
    R(x,U) - p = R(T^p x, U).
\end{align}
Indeed, for $n \in \N$, we see that $n \in R(x,U) - p$ if and only if $R(T^n x, U) = R(x,U) - n \in p$.  If $R(T^n x, U) \in p$, then
\[T^n T^p x = T^{n + p} x = T^{p + n} x = T^p T^n x \in \overline{U} = U,\]
whereby $n \in R(T^p x, U)$.  On the other hand, if $n \in R(T^p x, U)$, then $T^p T^n x = T^n T^p x \in U$, whereby $R(T^n x, U) \in p$.

Minimal left ideals of $\beta \N$ under addition by 1 are universal minimal dynamical systems.  More precisely, let $L \subseteq \beta \N$ be a minimal left ideal, and let $(X,T)$ be a not-necessarily-metrizable, minimal dynamical system.  Fix $x \in X$.  The map $T^{\displaystyle \cdot} x: \beta \N \to X$ described in the previous paragraph is continuous and satisfies $T^{\lambda_1 p} x = T^{1 + p} x = T (T^p x)$ and hence is a factor map $(\beta \N,\lambda_1) \to (X,T)$ of systems.  Restricted to $L$, by the minimality of $(X,T)$, we see that $T^{\displaystyle \cdot} x: (L, \lambda_1) \to (X,T)$ is surjective and hence is a factor map of minimal systems.

\subsection{Families of sets from dynamics and combinatorics}
\label{sec_families_of_sets}

In this subsection, we collect the families of subsets of $\N$ from dynamics and combinatorics relevant to this work.  If $\familytwo$ is a family, a set $A \subseteq \N$ is called a \emph{$\familytwo$ set} if it is a member of $\familytwo$, and a family $\familyone$ is called a \emph{$\familytwo$ family} if all members of $\familyone$ are $\familytwo$ sets, ie., $\familyone \subseteq \familytwo$. A syndetic filter, for example, is a filter every member of which is syndetic.

\subsubsection{Syndetic, thick, and piecewise syndetic sets}
\label{sec_syndetic_thick_ps_sets}

A set $A \subseteq \N$ is \dots
\begin{enumerate}
    \item \dots \emph{syndetic} if there exists $N \in \N$ such that
    \[
        A \cup (A-1) \cup \cdots \cup (A-N) = \N;
    \]

    \item \dots \emph{thick} if for all finite $F \subseteq \N$, there exists $n \in \N$ such that $F + n \subseteq A$;

    \item \dots \emph{piecewise syndetic} if there exists $N \in \N$ such that the set
    \[A \cup (A-1) \cup \cdots \cup (A-N) \text{ is thick}.\]
\end{enumerate}
The families of syndetic, thick, and piecewise syndetic subsets of $\N$ are denoted by $\syndetic$, $\thick$, and $\PS$, respectively. All three families are translation invariant, and
\[\PS = \syndetic \classcap \thick.\]
By \cref{lemma_class_cap_is_pr}, the family $\PS$ is piecewise syndetic and its dual, $\PS^*$, is a filter. 
Members of $\PS^*$ are frequently called \emph{thickly syndetic}: it is a short exercise to verify that $A \in \PS^*$ if and only if for all finite $F \subseteq \N$, there exists a syndetic set $S \subseteq \N$ such that $F + S \subseteq A$. In fact, the family of thickly syndetic sets is the largest syndetic, translation-invariant filter, as the following lemma shows.

\begin{lemma}[{\cite[Lemma 2.5]{glasscock_le_2024}}]
\label{lemma_condition_on_subfamily_of_ps_star}
    The family $\PS^*$ is a syndetic, translation-invariant filter.  If $\family$ is a syndetic, translation-invariant filter, then $\family \subseteq \PS^*$.
\end{lemma}

\subsubsection{Finite sums sets and central sets}
\label{sec_ip_and_central_sets}

A set $A \subseteq \N$ is called an \emph{IP set} if there exists a sequence $(x_i)_{i=1}^\infty \subseteq \N$ such that $A$ contains
\[
    \FS(x_i)_{i=1}^\infty \defeq \left\{ \sum_{f \in F} x_f \ \middle | \ F \subseteq \N \text{ is finite and nonempty} \right\}.
\]
Here ``FS'' stands for ``finite sums.'' 
Denote by $\IP$ the family of IP subsets of $\N$.
It is a consequence of Hindman's theorem \cite{hindman_1974} that the family $\IP$ is partition regular (see \cite[Lemma 2.1]{bergelson_hindman_1993}), and it is quick to show that thick sets are IP sets.
Thus, the dual family $\IP^*$ is a syndetic filter.

It is well known \cite[Thm. 5.12]{hindman_strauss_book_2012} that a set is an IP set if and only if it is member of an idempotent ultrafilter.  The following lemma offers a characterization in terms of idempotent filters, showing that the assumption of maximality (recall, ultrafilters are maximal filters) is superfluous. Since we do not need \cref{lemma_ip_set_iff_contained_in_prop_idemp_filter} in this paper, we omit its proof. 

\begin{lemma}
\label{lemma_ip_set_iff_contained_in_prop_idemp_filter}
    A set $A \subseteq \N$ is an IP set if and only if it is a member of a proper, idempotent filter.
\end{lemma}

We will consider the family of central sets for its importance to the subject and for its connection to the families of dynamically defined sets at the heart of this work.  Furstenberg \cite[Def. 8.3]{furstenberg_book_1981} defined a set $A \subseteq \N$ to be \emph{central} if there exists a system $(X,T)$, a uniformly recurrent point $x \in X$, a point $y \in X$ proximal to $x$, and an open set $U \subseteq X$ containing $x$ such that $A = R(y,U)$.  We denote the family of central sets by $\central$.

Bergelson, Hindman, and Weiss \cite{bergelson_hindman_1990} showed that a subset of $\N$ is central if and only if it is a member of a minimal idempotent ultrafilter on $\N$.  (This characterization of central sets fits in a series of related characterizations in more general settings: by Glasner \cite{glasner_1980} in countable, abelian groups; Bergelson-Hindman-Weiss \cite{bergelson_hindman_1990} in countable semigroups; and Shi-Yang \cite{shi_yang_1996} in all semigroups.) In analogy to IP sets and \cref{lemma_ip_set_iff_contained_in_prop_idemp_filter}, central sets can be characterized in terms of a special class of idempotent filters.  Defining this class -- ``collectionwise piecewise syndetic,'' idempotent filters -- would require too large a diversion, so we leave it to the interested reader. (See \cite[Def. 14.19]{hindman_strauss_book_2012} for a definition of collectionwise piecewise syndetic.)

It was stated by Furstenberg and follows immediately from the ultrafilter characterization of central sets that the family $\central$ is partition regular. It is well known that thick sets are central and that central sets are piecewise syndetic and IP.  Dual to the family of central sets is the family $\central^*$, a syndetic filter that will appear prominently in \cref{sec_dps_sets}. Figure \ref{fig:containmentdiagram} below depicts the relationships between the families of subsets of $\N$ introduced so far and those introduced later in the paper.

\begin{figure}[ht]
\centering
\begin{tikzpicture}[>=triangle 60]
  \matrix[matrix of math nodes,column sep={30pt,between origins},row sep={20pt,between origins}](m)
  {
    & & & & & & & & |[name=psstar]|\PS^* = \dPS^* & & & & & & \\
    \\
    & & & & & & |[name=thick]|\thick & &  \\
    & & & & & & & & |[name=dcpsstar]|\dcPS^* \\
    & & & & |[name=dthick]|\dthick & |[name=ipstar]|\IP^* & & & & & \\
    & & & & & & |[name=centralstar]|\central^* & & & & |[name=dcsyndetic]| \dcsyndetic \\
    \\
    & & & & |[name=dcthick]|\dcthick & & & & |[name=central]|\central \\
    & & & & & & & & & |[name=ip]|\IP & |[name=dsyndetic]| \dsyndetic \\
    & & & & & & |[name=dcps]| \dcPS & & & & \\
    & & & & & & & & |[name=syndetic]|\syndetic \\
    \\
    & & & & & & |[name=ps]|\PS = \dPS \\
 };
 
   \draw[-angle 90]
            (psstar) edge (thick)
            (psstar) edge (dcpsstar)
            (thick) edge (dthick)
            (dthick) edge (dcthick)
            (ipstar) edge (centralstar)
            (dcpsstar) edge (centralstar)
            (dcpsstar) edge (dcsyndetic)
            (centralstar) edge (dcthick)
            (thick) edge (central)
            (centralstar) edge (central)
            (dcsyndetic) edge (central)
            (dcsyndetic) edge (dsyndetic)
            (dcthick) edge (dcps)
            (central) edge (dcps)
            (central) edge (ip)
            (centralstar) edge (syndetic)
            (dsyndetic) edge (syndetic)
            (syndetic) edge (ps)
            (dcps) edge (ps)
  ;
\end{tikzpicture}
\caption{Containment amongst the families appearing in this paper.  An arrow $\family \to \familytwo$ indicates that $\family \subsetneq \familytwo$. 
There are eight families and their duals, yielding sixteen families in all, two pairs of which coincide. The diagram's symmetry is explained by the family duality described in \cref{sec_dual_families}.}
\label{fig:containmentdiagram}
\end{figure}
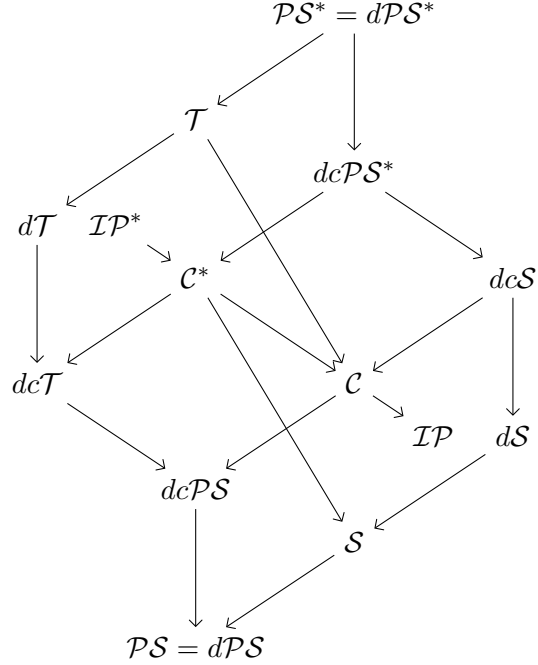

\subsection{Dynamically syndetic sets}
\label{sec_dyn_synd_facts}

In this subsection, we collect a number of results on the families of dynamically syndetic and dynamically central syndetic sets (recall their definitions from \cref{sec_intro_disjointness}). We denote these families by $\dsyndetic$ and $\dcsyndetic$, respectively. Note that these families are dual to the families of dynamically thick sets and sets of pointwise recurrence. 

\subsubsection{First results}

The first two results are quoted verbatim from our companion paper \cite{glasscock_le_2024}.  Since we use them several times, we give the full statements.

\begin{lemma}[{\cite[Lemma 3.2]{glasscock_le_2024}}]
    \label{lemma_ds_dcs_modifyable_on_a_finite_set}
    Let $A \subseteq \N$ be dynamically (central) syndetic.  If $B \subseteq \N$ is cofinite, then $A \cap B$ is dynamically (central) syndetic.
\end{lemma}

\begin{lemma}[{\cite[Lemma 3.3]{glasscock_le_2024}}]
\label{lemma_translates_of_dsyndetic_sets}
    Let $A \subseteq \N$ be dynamically syndetic.
    \begin{enumerate}
        \item
        \label{item_all_translates_are_ds}
        For all $n \in \N$, the set $A-n$ is dynamically syndetic.
        
        \item
        \label{item_some_translates_are_dcs}
        The set
        \begin{align}
            \label{eqn_translates_of_ds_set}
            \big\{ n \in \N \ \big| \ A-n \text{ is dynamically central syndetic} \big\}
        \end{align}
        is dynamically syndetic.
        
        \item
        \label{item_dcs_translates_of_dcs_are_dcs}
        If $A$ is dynamically central syndetic, then the set in \eqref{eqn_translates_of_ds_set} is dynamically central syndetic.
    \end{enumerate}
    Moreover, \eqref{item_all_translates_are_ds}, \eqref{item_some_translates_are_dcs}, and \eqref{item_dcs_translates_of_dcs_are_dcs} hold with $A-n$ replaced by $A+n$. Finally,
    \begin{align}
        \label{eqn_relationship_between_ds_and_dcs_classes}
        \dsyndetic = \bigcup_{n \in \N} \big(\dcsyndetic - n \big) = \bigcup_{n \in \N} \big(\dcsyndetic + n \big).
    \end{align}
\end{lemma}

The following lemma, well known to experts, describes the families of piecewise syndetic and central sets in dynamical terms.  Both equalities can be derived from the result of Auslander \cite{auslander_1960} and Ellis \cite{ellis_1957} that any point in any system is proximal to a uniformly recurrent point.  We opt for a short derivation of the first from the second, proven in \cite{Huang_Shao_Ye_2019}.

\begin{lemma}
\label{lem:central_characterization}
We have that $\PS = \dsyndetic \classcap \thick$ and $\central = \dcsyndetic \classcap \thick$.
\end{lemma}

\begin{proof}
    The second equality is proven in \cite[Thm. 3.7]{Huang_Shao_Ye_2019}.  We will derive the first from the second.  Suppose $A \subseteq \N$ is piecewise syndetic; it is not hard to see from the definition that there exists $N \in \N$ such that $A \cup (A+1) \cup \cdots \cup (A+N)$ is thick.  Since thick sets are central and the family of central sets is partition regular, some $A+i$ is central.  By the second equality, there exists $B \in \dcsyndetic$ and $H \in \thick$ such that $A+i \supseteq B \cap H$.  Therefore, we have that $A \supseteq (B-i) \cap (H-i)$.  By \cref{lemma_translates_of_dsyndetic_sets}, the set $B-i$ is dynamically syndetic and the set $H-i$ is thick, as was to be shown.
\end{proof}

\subsubsection{Characterizations of dynamical syndeticity}

The following theorems are the main results from our companion paper  \cite{glasscock_le_2024}.  They give characterizations of dynamical (central) syndeticity that will be key to several of the main results in this paper.

\begin{theorem}[{\cite[Theorem A, Theorem 3.7]{glasscock_le_2024}}]
\label{mainthm_ds_equivalents}
\label{thm:simple_equivalent_dS}
    Let $A \subseteq \N$.  The following are equivalent.
    \begin{enumerate}
        \item
        \label{item_intro_ds_def_condition}
        The set $A$ is dynamically syndetic.

        \item \label{item:1_A_uniformly_recurrent} There exists a nonempty subset $B \subseteq A$ for which the point $1_B$ is uniformly recurrent in the full shift $(\{0, 1\}^{\N},\sigma)$.  Moreover, when $B$ is considered as a subset of $\N_0$, the point $1_B$ is uniformly recurrent in the full shift $(\{0, 1\}^{\N_0},\sigma)$, the system $(X \defeq \overline{\sigma^{\N_0} 1_B}, \sigma)$ is minimal, the set $X \cap [1]_0$ is non-empty and clopen, and $B = R(1_{B}, X \cap [1]_0) \subseteq A$.

        \item
        \label{item_intro_ds_combo_condition}
        There exists a nonempty subset $B \subseteq A$ that satisfies: for all finite $F \subseteq B$, the set $\bigcap_{f \in F} (B-f)$ is syndetic.

        \item
        \label{item_intro_translate_belongs_to_sif}
        There exists $n \in \N$ for which the set $A - n$ belongs to a syndetic, idempotent filter on $\N$.
    \end{enumerate}
\end{theorem}

\begin{theorem}[{\cite[Theorem B, Theorem 3.6]{glasscock_le_2024}}]
    \label{intro_mainthm_full_chars_of_dcsyndetic}
    \label{thm:simple_equivalent_dcS}
    Let $A \subseteq \N$.  The following are equivalent.
    \begin{enumerate}
        \item \label{intro_item_def_of_dcs}
        The set $A$ is dynamically central syndetic.

        \item \label{item:very_strongly_central}
        There exists a minimal system $(X,T)$, a nonempty, open set $U \subseteq X$, and a point $x \in \overline{U}$ such that $R(x,U) \subseteq A$.

        \item \label{item:1_A_0_uniformly_recurrent} There exists a subset $B \subseteq A$ such that the point $1_{B \cup \{0\}}$ is uniformly recurrent in the full shift $(\{0, 1\}^{\N_0}, \sigma)$. In particular, the system $(X \defeq \overline{\sigma^{\N_0} 1_{B \cup \{0\}}}, \sigma)$ is minimal, the set $X \cap [1]_0$ is a clopen neighborhood of $1_{B \cup \{0\}}$, and $B = R(1_{B \cup \{0\}}, X \cap [1]_0) \subseteq A$.

        \item \label{intro_item_dcs_combo_condition}
        There exists a subset $B \subseteq A$ that satisfies: for all finite $F \subseteq B$, the set $B \cap \bigcap_{f \in F} (B-f)$ is syndetic.

        \item \label{intro_item_member_of_sif}
        The set $A$ belongs to a syndetic, idempotent filter on $\N$.
    \end{enumerate}
\end{theorem}

\section{Characterizations of dynamically thick sets and sets of pointwise recurrence}
\label{sec_dthick}

In this section, we study the families of dynamically thick sets and sets of pointwise recurrence.  The main results are the combinatorial characterizations of these families in \cref{sec_dthick_comb_characterizations} (Theorems \ref{mainthm_dct_combo_characterizations} and \ref{mainthm_dt_characterizations}); an example of a dynamically thick set which is not an IP set in \cref{sec_examples_using_combinatorics} (\cref{Mainthm:dyn_thick_not_IP}); a  structural result for dynamically thick sets in \cref{sec_general_form_of_dthick} (\cref{mainthm_structure_for_dy_thick_sets}); and that the family of dynamically syndetic sets are not $\sigma$-compact in \cref{sec:sigma-compact} (\cref{theorem_dy_syndetic_is_not_sigma_compact}).

\subsection{Definitional robustness}
\label{sec:definitional_robustness}

The following lemma demonstrates that metrizability is not an essential feature in the definitions of dynamical syndeticity or dynamical thickness.

\begin{lemma}
\label{lemma_metrizability_doesnt_matter}
Let $(X,T)$ be a not-necessarily-metrizable, minimal topological dynamical system.  For all $x \in X$ and all nonempty, open $U \subseteq X$ (containing $x$), the set $R(x,U)$ is dynamically (central) syndetic.
\end{lemma}

\begin{proof}
    Let $x \in X$ and $U \subseteq X$ be nonempty and open.  Fix $L \subseteq \beta \N$ a minimal left ideal.  As described in \cref{sec_topology_and_ultrafilters}, the map $\pi: (L, \lambda_1) \to (X,T)$ given by $\pi(p) = T^p x$ is a factor map of minimal systems.  By \cite[Thm. 19.23]{hindman_strauss_book_2012}, since $x$ is uniformly recurrent, there exists an idempotent $p \in L$ such that $\pi(p) = x$. It is quick to check that
    \[R_T(x, U) = R_{\lambda_1}(p, \pi^{-1} U).\]
    We will show that the set $R_{\lambda_1}(p, \pi^{-1} U)$ is dynamically syndetic.

    By the definition of the topology on $\beta \N$, since $\pi^{-1} U$ is nonempty and open, there exists $A \subseteq \N$ such that $\emptyset \neq \overline{A} \cap L \subseteq \pi^{-1} U$.  In the notation defined in \cref{sec_family_algebra}, we see that
    \begin{align}
    \label{eqn_requisite_containment_of_non_metrizable_returns}
        R_{\lambda_1}(p, \pi^{-1} U) \supseteq R_{\lambda_1}(p, \overline{A}) = A - p.
    \end{align}
    Since $\overline{A} \cap L \neq \emptyset$ and $(L,\lambda_1)$ is minimal, we see that the set $A-p$ is nonempty (indeed, syndetic).

    Considering $1_A$ as an element of $\{0,1\}^{\N_0}$, we see in the full shift $(\{0,1\}^{\N_0}, \sigma)$ that $A = R_{\sigma}(1_A,[1]_0)$.  It follows by \eqref{eqn_ultrafilter_shift_of_return_set} in \cref{sec_topology_and_ultrafilters} that
    \[A - p = R_{\sigma}(1_A,[1]_0) - p = R_{\sigma}(\sigma^p 1_A, [1]_0).\]
    Since $p$ is a minimal idempotent ultrafilter, we have by \cite[Thm. 19.23]{hindman_strauss_book_2012} that the point $T^p 1_A$ is uniformly recurrent under the shift.  Since $R_{\sigma}(\sigma^p 1_A, [1]_0)$ is nonempty, we see that the set $R_{\sigma}(\sigma^p 1_A, [1]_0)$, and hence the set $A - p = R_{\lambda_1}(p, \overline{A})$, is dynamically syndetic.  Because these families are upward closed, by \eqref{eqn_requisite_containment_of_non_metrizable_returns}, the same can be said about the set $R_{\lambda_1}(p, \pi^{-1} U) = R_T(x, U)$, as was to be shown.

    In the case that $x \in U$, we wish to show that the set $R_T(x,U)$ is dynamically central syndetic.  The argument given above can be followed with some additional information.  First, since $x \in U$, the point $p \in \pi^{-1} U$.  The set $A \subseteq \N$ can then be assumed to be such that $p \in \overline{A} \cap L$. Considering $1_A$ as an element of $\{0,1\}^{\N_0}$, we see
    \[R_\sigma(1_A,[1]_0) = A \in p,\]
    whereby $\sigma^p 1_A \in [1]_0$.  Thus, the set $A-p = R_{\sigma}(\sigma^p 1_A, [1]_0)$ is dynamically central syndetic, as desired.
\end{proof}

We now gather results showing that the families of dynamically thick sets and sets of pointwise recurrence are robust under minor changes to the combinatorial, topological, and dynamical requirements in their definitions.  A similar robustness for dynamically (central) syndetic sets was discussed in \cite[Section 3]{glasscock_le_2024}.  Thus, the families of sets of pointwise recurrence and dynamically thick subsets of $\N$ remain unchanged if, in their definitions, \dots
\begin{enumerate}
    \item \dots \ ``for all minimal systems'' is changed to ``for all invertible minimal systems,'' ``for all zero dimensional minimal systems ,'' or ``for all (not-necessarily-metrizable) minimal systems.'' By duality, these follow immediately from \cite[Lemma 3.1]{glasscock_le_2024}, \cref{thm:simple_equivalent_dS} \eqref{item:1_A_uniformly_recurrent}, \cref{thm:simple_equivalent_dcS} \eqref{item:1_A_0_uniformly_recurrent}, and \cref{lemma_metrizability_doesnt_matter} below, respectively.

    \item \dots \ (for sets of pointwise recurrence) ``all open $U \subseteq X$ with $x \in U$'' is changed to ``all open $U \subseteq X$ with $x \in \overline{U}$.'' This is a consequence of \cref{thm:simple_equivalent_dcS} \eqref{item:very_strongly_central}.

    \item \dots \ sets are considered equivalent up to a non-piecewise syndetic set.  This is a consequence of \cref{thm:remove_nonPS_from_dynthick}: a set $A \subseteq \N$ is a set of pointwise recurrence (dynamically thick) if and only if for all non-piecewise syndetic set $B \subseteq \N$, the set $A \setminus B$ is a set of pointwise recurrence (dynamically thick).
\end{enumerate}

\subsection{First results}
\label{sec_dt_first_results}

In this subsection, we collect some preliminary results concerning dynamical thickness.  As discussed in \cref{sec_intro_set_and_pt_recurrence}, that sets of pointwise recurrence are piecewise syndetic has been shown several times in the literature. We record it here for ease of reference.

\begin{lemma}[\cite{dong_shao_ye_2012, Glasner-Weiss-Interpolation, huang_ye_2005}]
\label{thm_dy_c_thick_is_ps}
    Sets of pointwise recurrence (and, hence, dynamically thick sets) are piecewise syndetic.
\end{lemma}

\begin{lemma}
    \label{lemma_intersection_of_ds_and_dt_is_infinite}
    The intersection of a dynamically (central) syndetic set and a dynamically thick set (set of pointwise recurrence) is infinite.
\end{lemma}

\begin{proof}
    Let $A \in \dsyndetic$, $B \in \dthick$, and $n \in \N$.  By \cref{lemma_ds_dcs_modifyable_on_a_finite_set}, we see that $A \cap \{n+1, n+2, \ldots\} \in \dsyndetic$.  Since $B \in \dthick$, we have that $A \cap \{n+1, n+2, \ldots\} \cap B \neq \emptyset$.  Therefore, $\max (A \cap B) > n$.  Since $n \in \N$ was arbitrary, we see that $A \cap B$ is infinite.  The same argument works to show that $A \cap B$ is infinite when $A \in \dcsyndetic$ and $B \in \dcthick$.
\end{proof}

\cref{thm:relations_C_dPS_PS-intro} in the next section simultaneously strengthens both \cref{thm_dy_c_thick_is_ps} and \cref{lemma_intersection_of_ds_and_dt_is_infinite} by showing that the intersection of a dynamically (central) syndetic set and a dynamically thick set (set of pointwise recurrence) is, in fact, piecewise syndetic.

\begin{lemma}
\label{lemma_translates_of_dthick_sets}
    Let $A \subseteq \N$.
    \begin{enumerate}
        \item
        \label{item_all_translates_are_dt}
        If $A$ is dynamically thick, then for all $n \in \N$, the set $A-n$ is dynamically thick.
        
        \item
        \label{item_some_translates_are_dct}
        If the set
        \begin{align}
            \label{eqn_translates_of_dt_set}
            \big\{ n \in \N \ \big| \ A-n \text{ is a set of pointwise recurrence} \big\}
        \end{align}
        is dynamically thick, then $A$ is dynamically thick.
        
        \item
        \label{item_dcs_translates_of_dct_are_dct}
        If the set in \eqref{eqn_translates_of_dt_set} is a set of pointwise recurrence, then $A$ is a set of pointwise recurrence.
    \end{enumerate}
    Moreover, \eqref{item_all_translates_are_dt}, \eqref{item_some_translates_are_dct}, and \eqref{item_dcs_translates_of_dct_are_dct} hold with $A-n$ replaced by $A+n$. Finally,
    \begin{align}
        \label{eqn_relationship_between_dt_and_dct_classes}
        \dthick = \bigcap_{n \in \N} \big(\dcthick - n \big) = \bigcap_{n \in \N} \big(\dcthick + n \big).
    \end{align}
\end{lemma}

\begin{proof}
    \eqref{item_all_translates_are_dt} \ Suppose $A \in \dthick$ and $n \in \N$. To see that $A - n \in \dthick$, it suffices to show that $A-n \in \dsyndetic^*$.  Let $B \in \dsyndetic$.  By \cref{lemma_translates_of_dsyndetic_sets}, the set $B + n \in \dsyndetic$.  By \cref{lemma_intersection_of_ds_and_dt_is_infinite}, the set $A \cap (B+n)$ is infinite, so $(A-n) \cap B \supseteq (A-n) \cap (B+n - n) \neq \emptyset$, as desired.

    \eqref{item_some_translates_are_dct} \ In the language of the family algebra developed in \cref{sec_combinatorics}, this statement is equivalent to $\dthick + \dcthick \subseteq \dthick$.  Recall from \cref{lemma_translates_of_dsyndetic_sets} that $\dsyndetic \subseteq \dsyndetic + \dcsyndetic$. Taking the dual and simplifying using the algebra in \cref{lemma_fam_dual_sums}, we see that $\dthick + \dcthick \subseteq \dthick$.

    \eqref{item_dcs_translates_of_dct_are_dct} \ This statement is equivalent to $\dcthick + \dcthick \subseteq \dcthick$. Taking the dual of $\dcsyndetic \subseteq \dcsyndetic + \dcsyndetic$ from \cref{lemma_translates_of_dsyndetic_sets}, we see that $\dcthick + \dcthick \subseteq \dcthick$.
    
    That \eqref{item_all_translates_are_dt} holds with $A-n$ replaced with $A+n$ follows the same argument used in \eqref{item_all_translates_are_dt}.  That \eqref{item_some_translates_are_dct} and \eqref{item_dcs_translates_of_dct_are_dct} hold with $A-n$ replaced by $A+n$ requires different reasoning, since the set algebra explanation only treats negative translates.

    Suppose that $A$ is such that the set in \eqref{eqn_translates_of_dt_set} (with $A+n$ instead of $A-n$) is dynamically thick.  We wish to show that $A$ is dynamically thick. It suffices by the definitional robustness discussed in the previous subsection to show: for all minimal, invertible $(X,T)$, all $x \in X$, and all nonempty, open $U \subseteq X$, there exists $a \in A$ such that $T^a x \in U$.  Thus, let $(X,T)$ be a minimal, invertible system, $x \in X$, and $U \subseteq X$ open and nonempty.  Since the system $(X,T^{-1})$ is minimal (see, for example, \cite[Lemma 2.7]{glasscock_koutsogiannis_richter_2019}) and the set in \eqref{eqn_translates_of_dt_set} is dynamically thick, there exists $n \in \N$ such that $T^{-n} x \in U$ and $A + n \in \dcthick$.  By the definition of set of pointwise recurrence, there exists $m \in A+n$ (so $m > n$) such that $T^m T^{-n} x \in U$.  We see that $m - n \in A$ and $T^{m-n} x \in U$, as desired.  The argument that \eqref{item_dcs_translates_of_dct_are_dct} holds with $A-n$ replaced by $A+n$ is very similar and is left to the interested reader.

    Finally, that \eqref{eqn_relationship_between_dt_and_dct_classes} holds follows immediately from \eqref{item_all_translates_are_dt} and \eqref{item_some_translates_are_dct} for $A-n$ and $A+n$.
\end{proof}

As demonstrated in the proof of \cref{lemma_translates_of_dthick_sets}, it is interesting and useful to formulate the lemma's conclusions in terms of the family algebra developed in \cref{sec_combinatorics}. Thus,
\begin{enumerate}
    \item the family $\dthick$ is translation invariant, that is, $\dthick \subseteq \{\N\} + \dthick$;
    \item $\dthick + \dcthick \subseteq \dthick$;
    \item $\dcthick + \dcthick \subseteq \dcthick$.
\end{enumerate}

\begin{lemma}
\label{lem:dilates_of_dT_dcT}
    Let $A \subseteq \N$ and $k \in \N$. If $A$ is a set of pointwise recurrence, then so are the sets $kA$ and $A/k$.  If $A$ is dynamically thick, then so is the set $A/k$.
\end{lemma}

\begin{proof}
    We will make use of the analogous result for dynamically (central) syndetic sets from \cite[Lemma 3.4]{glasscock_le_2024}: \emph{Let $k \in \N$. If $B \subseteq \N$ is dynamically central syndetic, then so are the sets $kB$ and $B/k$.  If $B$ is dynamically syndetic, then so is the set $kB$.}
    
    Suppose $A \in \dcthick$.  If $B \in \dcsyndetic$, then the sets $kB$ and $B/k$ are $\dcsyndetic$ sets, whereby $A \cap (kB) \neq \emptyset$ and $A \cap (B/k) \neq \emptyset$.  Since $(kB)/k = B$ and $k(B/k) \subseteq B$, it follows that $A/k \cap B \neq \emptyset$ and $kA \cap B \neq \emptyset$. Since $B \in \dcsyndetic$ was arbitrary, we have that $A/k$ and $kA$ are sets of pointwise recurrence.

    When $A \in \dthick$, the same argument shows that $A/k \in \dthick$.
\end{proof}

\subsection{Combinatorial characterizations: proofs of Theorems \ref{mainthm_dct_combo_characterizations} and \ref{mainthm_dt_characterizations}}
\label{sec_dthick_comb_characterizations}

In this subsection, we give combinatorial characterizations of dynamical thickness and pointwise recurrence.  
Theorems \ref{mainthm_dct_combo_characterizations} and \ref{mainthm_dt_characterizations} will follow immediately from Theorems \ref{thm_dct_combo_characterizations} and \ref{lemma_our_fav_dy_thick} below, respectively.  In each, the second statement comes from interpreting ``has nonempty intersection with all dynamically (central) syndetic sets'' in combinatorial terms by appealing to the main results on dynamically syndetic sets from \cite{glasscock_le_2024}, while the third, fourth, and fifth statements arise from ultrafilter dynamics, appealing ultimately to \cref{lemma_metrizability_doesnt_matter}.

\begin{theorem}
\label{thm_dct_combo_characterizations}
    Let $A \subseteq \N$.  The following are equivalent.
    \begin{enumerate}
        \item
        \label{item_dct_def_char}
        The set $A$ is a set of pointwise recurrence.

        \item
        \label{item_dct_sif_char}
        For all $B \supseteq A$, there exists a finite set $F \subseteq \N \setminus B$ such that the set $B \cup \big( B-F \big)$ is thick.

        \item
        \label{item_condition_one_dct}
        For all syndetic $S \subseteq \N$, there exists a finite set $F \subseteq A$ such that for all syndetic $S' \subseteq S$, the set $F \cap (S - S')$ is nonempty.

        \item \label{item:complement_thick_dct_characterization_v2}
        For all syndetic $S \subseteq \N$ and all thick $H \subseteq \N$, the piecewise syndetic set $P \defeq S \cap H$ satisfies the following.  There exists a finite set $F \subseteq A$ such that for all syndetic $S' \subseteq S$, the piecewise syndetic set $P' \defeq S' \cap H$ is such that the set $P' \cap (P - F)$ is nonempty.
        
        \item
        \label{item:complement_psstar_dct_characterization_star_fixed}
        For all piecewise syndetic $S \subseteq \N$ and all minimal left ideals $L \subseteq \beta \N$ with $\overline{S} \cap L \neq \emptyset$, there exists a finite set $F \subseteq A$ such that $\overline{S} \cap L \subseteq \overline{S-F}$.
    \end{enumerate}
\end{theorem}

\begin{proof}
    \eqref{item_dct_def_char} $\iff$ \eqref{item_dct_sif_char} \ Because $\dcsyndetic$ and $\dcthick$ are dual, the set $A$ is a set of pointwise recurrence if and only if the set $A' \defeq \N \setminus A$ is not dynamically central syndetic.  By \cref{intro_mainthm_full_chars_of_dcsyndetic}, the set $A'$ is not dynamically central syndetic if and only if for all $B' \subseteq A'$, there exists a finite set $F \subseteq B'$ such that the set
    \[B' \cap \bigcap_{f \in F} (B' - f)\]
    is not syndetic.  Taking the complement of this set in $\N$, we see that this happens if and only if the set
    \[\big( \N \setminus B' \big) \cup \bigcup_{f \in F} \big( (\N \setminus B') - f \big) = (\N \setminus B') \cup \big( (\N \setminus B')  - F \big)\]
    is thick.  Set $B \defeq \N \setminus B'$. Note that $B' \subseteq A'$ if and only if $A \subseteq B$.
    
    Summarizing the previous paragraph, we have shown that the set $A$ is a set of pointwise recurrence if and only if for all $B \supseteq A$, there exists a finite set $F \subseteq \N \setminus B$ such that the set $B \cup (B - F)$ is thick, as was to be shown.

    \eqref{item_dct_def_char} $\implies$ \eqref{item:complement_psstar_dct_characterization_star_fixed} \ Suppose that $A \in \dcthick$.  Let $S \subseteq \N$ be piecewise syndetic and $L \subseteq \beta \N$ be a minimal left ideal such that $\overline{S} \cap L \neq \emptyset$.  Note that $\overline{S} \cap L$ is nonempty and clopen in $L$.  By the definitional robustness of $\dcthick$ sets regarding non-metrizable systems (see \cref{lemma_metrizability_doesnt_matter}), for all $p \in \overline{S} \cap L$, there exists $a \in A$ such that $a + p \in \overline{S}$. Therefore,
    \[\overline{S} \cap L \subseteq \bigcup_{a \in A} \overline{S - a}.\]
    Because $\overline{S} \cap L$ is compact, there exists a finite set $F \subseteq A$ such that
    \[\overline{S} \cap L \subseteq \bigcup_{f \in F} \overline{S - f} \subseteq \overline{S-F},\]
    as was to be shown.

    \eqref{item:complement_psstar_dct_characterization_star_fixed} $\implies$ \eqref{item:complement_thick_dct_characterization_v2} \ Let $S \subseteq \N$ be syndetic and $H \subseteq \N$ be thick, and put $P \defeq S \cap H$.  It follows by \cite[Thm. 4.48]{hindman_strauss_book_2012} that there exists a minimal left ideal $L \subseteq \overline{H}$.  Note that since $S$ is syndetic, we have that $\emptyset \neq \overline{S} \cap L = \overline{S} \cap L \cap \overline{H} = \overline{P} \cap L$.
    
    Let $F \subseteq A$ be the finite set guaranteed by \eqref{item:complement_psstar_dct_characterization_star_fixed} for the piecewise syndetic set $P$.  Let $S' \subseteq S$ be syndetic, and put $P' \defeq S' \cap H$. Since $S'$ is syndetic and $S' \subseteq S$, we have that $\emptyset \neq \overline{S'} \cap L = \overline{P'} \cap L$.  Since $\emptyset \neq \overline{P'} \cap L \subseteq \overline{P} \cap L \subseteq \overline{P-F}$, we see that $\overline{P'} \cap \overline{P-F} \neq \emptyset$, whereby $P' \cap (P-F) \neq \emptyset$, as desired.
    
    \eqref{item:complement_thick_dct_characterization_v2} $\implies$ \eqref{item_condition_one_dct} \ Let $S \subseteq \N$ be syndetic.  Get $F \subseteq A$ finite from \eqref{item:complement_psstar_dct_characterization_star_fixed} for $H = \N$.  Then for all $S' \subseteq S$ syndetic,
    \[S' \cap (S-F) \neq \emptyset,\]
    whereby $F \cap (S-S') \neq \emptyset$, as desired.

    \eqref{item_condition_one_dct} $\implies$ \eqref{item_dct_def_char} \ We will prove the contrapositive: if $A$ is not a set of pointwise recurrence, then there exists a syndetic set $S \subseteq \N$ such that for all finite $F \subseteq A$, there exists a syndetic set $S' \subseteq S$ such that $F \cap (S - S') = \emptyset$. Since $F \cap (S - S') = \emptyset$ if and only if $S' \cap (S - F) = \emptyset$, it is equivalent to show: there exists a syndetic set $S \subseteq \N$ such that for all finite $F \subseteq A$, the set $S \cap (\N \setminus (S-F))$ is syndetic.  Taking complements and considering $B$ as $\N \setminus A$ and $C$ as $\N \setminus S$, we must show: if a set $B \subseteq \N$ is dynamically central syndetic, then there exists a set $C \subseteq \N$ that is not thick such that for all finite $F \subseteq \N \setminus B$, the set $(\N \setminus C) \cap \bigcap_{f \in F} (C - f)$ is syndetic.

    Suppose $B \subseteq \N$ is dynamically central syndetic.  If $B = \N$, put $C = \emptyset$ and note that the conclusion holds since the empty intersection is $\N$. Suppose $B \neq \N$. By \cref{thm:simple_equivalent_dcS}, there exists a minimal system $(X,T)$, a point $x \in X$, and a clopen set $U \subseteq X$ with $x \in U$ such that $R(x,U) \subseteq B$.  Define $V \defeq X \setminus U$ and $C \defeq \N \setminus R(x,U) = R(x, V)$. Since $U \neq X$ and the set $R(x,U)$ is syndetic, the set $C$ is nonempty but not thick.  Let $F \subseteq \N \setminus B$ be finite.  Since $F \subseteq C$, we have that for all $f \in F$, the set $T^{-f}V$ is an open neighborhood of $x$.  We see that
    \[\big(\N \setminus C\big) \cap \bigcap_{f \in F} (C - f) = R\bigg( x, U \cap \bigcap_{f \in F}T^{-f}V \bigg),\]
    which is syndetic since the set $U \cap \bigcap_{f \in F}T^{-f}V$ is an open neighborhood of $x$, as desired.
\end{proof}

\begin{theorem}
\label{lemma_our_fav_dy_thick}
    Let $A \subseteq \N$.  The following are equivalent.
    \begin{enumerate}
        \item 
        \label{item_dthick_def}
        The set $A$ is dynamically thick.

        \item 
        \label{item:dyn_thick_idempotent}
        For all $\N \supsetneq B \supseteq A$, there exists a finite set $F \subseteq \N \setminus B$ such that the set $B-F$ is thick.

        \item \label{item_condition_one}
        For all syndetic $S \subseteq \N$, there exists a finite set $F \subseteq A$ such that for all syndetic $S' \subseteq \N$, the set $F \cap (S - S')$ is nonempty.

        \item \label{item:complement_thick_dt_characterization_v2}
        For all piecewise syndetic $S \subseteq \N$, there exists a finite set $F \subseteq A$ such that the set $S-F$ is thick.

        \item
        \label{item:complement_psstar_dt_characterization_star_fixed}
        For all piecewise syndetic $S \subseteq \N$ and all minimal left ideals $L \subseteq \beta \N$ with $\overline{S} \cap L \neq \emptyset$, there exists a finite set $F \subseteq A$ such that $L \subseteq \overline{S-F}$.
    \end{enumerate}
\end{theorem}

\begin{proof}
    \eqref{item_dthick_def} $\iff$ \eqref{item:dyn_thick_idempotent}
    Because $\dsyndetic$ and $\dthick$ are dual, the set $A$ is dynamically thick if and only if the set $A' \defeq \N \setminus A$ is not dynamically syndetic.  By \cref{mainthm_ds_equivalents}, the set $A'$ is not dynamically syndetic if and only if for all nonempty subsets $B' \subseteq A'$, there exists a finite set $F \subseteq B'$ such that the set $\bigcap_{f \in F} (B' - f)$ is not syndetic.  Taking the complement of this set in $\N$, we see that this happens if and only if the set
    \[ \bigcup_{f \in F} \big( (\N \setminus B') - f \big) = \big( \N \setminus B' \big) - F\]
    is thick.  Set $B \defeq \N \setminus B'$. Note that $B' \subseteq A'$ if and only if $A \subseteq B$ and that $B' \neq \emptyset$ if and only if $B \neq \N$.
    
    Summarizing the previous paragraph, we have shown that the set $A$ is dynamically thick if and only if for all $A \subseteq B \subsetneq \N$, there exists a finite set $F \subseteq \N \setminus B$ such that the set $B - F$ is thick, as was to be shown.

    \eqref{item_dthick_def} $\implies$ \eqref{item:complement_psstar_dt_characterization_star_fixed} \ Suppose that $A \in \dthick$.  Let $S \subseteq \N$ be piecewise syndetic and $L \subseteq \beta \N$ be a minimal left ideal such that $\overline{S} \cap L \neq \emptyset$.  Note that $\overline{S} \cap L$ is nonempty and open in $L$. By the definitional robustness of $\dthick$ sets regarding non-metrizable systems (see \cref{lemma_metrizability_doesnt_matter}), for all $p \in L$, there exists $a \in A$ such that $a + p \in \overline{S}$.  Therefore,
    \[L \subseteq \bigcup_{a \in A} \overline{S - a}.\]
    Because $L$ is compact, there exists a finite set $F \subseteq A$ such that
    \[L \subseteq \bigcup_{f \in F} \overline{S - f}  \subseteq \overline{S-F},\]
    as was to be shown.

    \eqref{item:complement_psstar_dt_characterization_star_fixed} $\implies$ \eqref{item:complement_thick_dt_characterization_v2} \ Let $S \subseteq \N$ be a piecewise syndetic set.  It follows by \cite[Thm. 4.40]{hindman_strauss_book_2012} that there exists a minimal left ideal $L \subseteq \beta \N$ such that $\overline{S} \cap L \neq \emptyset$.  Let $F \subseteq A$ be the finite set guaranteed by \eqref{item:complement_psstar_dt_characterization_star_fixed}.  Since $\overline{S-F}$ contains a minimal left ideal, by \cite[Thm. 4.48]{hindman_strauss_book_2012}, the set $S-F$ is thick, as desired.

    \eqref{item:complement_thick_dt_characterization_v2} $\implies$ \eqref{item_condition_one} \ Let $S \subseteq \N$ be syndetic.  Let $F \subseteq A$ be the finite set guaranteed by \eqref{item:complement_thick_dt_characterization_v2}.  Since $S-F$ is thick, for all $S' \subseteq \N$ syndetic, the set $S' \cap (S-F)$ is nonempty, whereby $F \cap (S-S') \neq \emptyset$, as desired.

    \eqref{item_condition_one} $\implies$ \eqref{item_dthick_def} \ We will prove the contrapositive: if $A$ is not dynamically thick, then there exists a syndetic set $S \subseteq \N$ such that for all finite $F \subseteq A$, the set $S - F$ is not thick. Taking complements and considering $B$ as $\N \setminus A$ and $C$ as $\N \setminus S$, we must show: if a set $B \subseteq \N$ is dynamically syndetic, then there exists a set $C \subseteq \N$ that is not thick such that for all finite $F \subseteq \N \setminus B$, the set $\bigcap_{f \in F} (C - f)$ is syndetic.

    Suppose $B \subseteq \N$ is dynamically syndetic. If $B = \N$, put $C = \emptyset$ and note that the conclusion holds since the empty intersection is $\N$. Suppose $B \neq \N$. By \cref{thm:simple_equivalent_dS}, there exists a minimal system $(X,T)$, a point $x \in X$, and a nonempty, clopen set $U \subseteq X$ such that $R(x,U) \subseteq B$.  Define $V \defeq X \setminus U$ and $C \defeq \N \setminus R(x,U) = R(x, V)$. Since $U \neq X$ and the set $R(x,U)$ is syndetic, the set $C$ is nonempty but not thick.  Let $F \subseteq \N \setminus B$ be finite.  Since $F \subseteq C$, we have that for all $f \in F$, the set $T^{-f}V$ is an open neighborhood of $x$.  We see that
    \[\bigcap_{f \in F} (C - f) = R\bigg( x, \bigcap_{f \in F}T^{-f}V \bigg),\]
    which is syndetic since the set $\bigcap_{f \in F}T^{-f}V$ is an open neighborhood of $x$, as desired.
\end{proof}

\section{The structure of dynamically thick sets}

\label{sec:structure_dyn_thick}

We begin this section with some examples of dynamically thick sets and a proof of \cref{Mainthm:dyn_thick_not_IP}. In \cref{sec_dy_thick_from_disjointness} and \cref{sec_dy_thick_from_distal}, we develop more sophisticated dynamical tools to generalize the examples in Propositions \ref{ex:type2-special} and \ref{rem:combinatorial_primes}.  We conclude with a proof of \cref{mainthm_structure_for_dy_thick_sets} in \cref{sec_general_form_of_dthick}, which shows that every dynamically thick set takes the form exhibited in these examples.

\subsection{Examples of dynamically thick sets via the combinatorial characterizations and proof of \texorpdfstring{\cref{Mainthm:dyn_thick_not_IP}}{Theorem C}}

\label{sec_examples_using_combinatorics}

In this subsection, we will use the combinatorial characterizations from \cref{mainthm_dt_characterizations} to give two concrete examples of dynamically thick sets.

\begin{proposition}
\label{ex:type2-special}
    Let $k \in \N$ and $(H_i)_{i \in \N}$ be a sequence of thick subsets of $\N$. The set
    \[A \defeq \bigcup_{i=0}^{k-1} \big ((k \N  + i) \cap H_i \big)\]
    is dynamically thick.
\end{proposition}

\begin{proof}
    We will apply \cref{mainthm_dt_characterizations} \eqref{item:intro_complement_thick_dt_characterization_v2}.  Let $S \subseteq \N$ be piecewise syndetic.  Since piecewise syndeticity is partition regular and $S = \bigcup_{i=0}^{k-1} \big(S \cap (k\N - i)\big)$, there exists $i \in \{0, \ldots, k-1\}$ such that $S \cap (k \N - i)$ is piecewise syndetic.  Since demonstrating \eqref{item:intro_complement_thick_dt_characterization_v2} for a subset of $S+i$ suffices to demonstrate it for $S$, by replacing $S$ with $(S+i) \cap k\N$, we can proceed under the assumption that $S \subseteq k \N$.

    Write $S = B \cap C$ where $B$ is syndetic with maximum gap length less than $\ell \in \N$ and $C$ is thick. For $i \in \{0, \ldots, k-1\}$, choose an interval $I_i$ of $H_i$ such that $|I_i| > \ell$ and $\max I_i < \min I_{i+1}$. The set 
    \[
        F \defeq \bigcup_{i=0}^{k-1} \big( (k\N + i) \cap I_i \big)
    \]
    is a finite subset of $A$. 
    We will show that for all $M \in \N$, there is a finite subset $S' \subseteq S$ for which the set $S' - F$ contains an interval of length greater than $M$. It will follow that the set $S - F$ is thick.

    Let $M \in \N$.  Choose $S' \subseteq S$ to be the set $S$ intersected with a long interval (whose length will be specified later) on which the distance between consecutive elements of $S$ is less than $\ell$.  Since $\ell < |I_i|$ and $S' \subseteq k\N$, the set $S' - \big((k\N + i) \cap I_i \big)$ contains the set $(k \N - i) \cap J_i$, where
    \[
        J_i \defeq \big\{\min S' - \max I_i, \ \ldots, \max S' - \min I_i \big\}.
    \]
    It follows that
    \[
        S' - F \supseteq \bigcup_{i=0}^{k-1} \big((k \N - i) \cap J_i\big) \supseteq \bigcap_{i=0}^{k-1} J_i = \big\{\min S' - \max I_0, \ \ldots, \max S' - \min I_{k-1} \big\},
    \]
    provided $\min S' - \max I_0 \leq \max S' - \min I_{k-1}$.  This interval has length $\max S' - \min S' - (\min I_{k-1} - \max I_0)$ which is greater than $M$ if $\max S' - \min S' > M + \min I_{k-1} - \max I_0$.
\end{proof}

The second example uses \cref{mainthm_dt_characterizations} \eqref{item:intro_dyn_thick_idempotent} and gives a proof of  \cref{Mainthm:dyn_thick_not_IP}.

\begin{proposition}
\label{rem:combinatorial_primes}
    Let $(p_i)_{i \in \N} \subseteq \N$ be a sequence of distinct primes, $(c_i)_{i \in \N} \subseteq \Z$ be a sequence of integers, and $(H_i)_{i \in \N}$ be a sequence of thick subsets of $\N$. The set
    \[
        A \defeq \bigcup_{i=1}^\infty \big ((p_i \N  + c_i) \cap H_i \big)
    \]
    is dynamically thick.  
    
    Moreover, if for all $n \in \N$, we have that $c_n \not\equiv 0 \bmod {p_n}$ and, for all but finitely many distinct pairs $i, j$ of positive integers, the set $H_i \cap (H_j - n)$ is empty, then the set $A$ is not an IP set.
\end{proposition}

\begin{proof}
    To apply \cref{mainthm_dt_characterizations} \eqref{item:intro_dyn_thick_idempotent}, let $\N \supsetneq B \supseteq A$. We want to show there exists a finite set $F \subseteq \N \setminus B$ such that $B - F$ is thick. Consider two cases.

    Case 1: There exists $i \in \N$ for which the set $\N \setminus B$ contains a complete residue system modulo $p_i$. Let $F = \{f_0, \ldots, f_{p_i - 1}\} \subseteq \N \setminus B$ be a complete system of modulo $p_i$ residues. We see that
    \[B - F \supseteq A - F \supseteq \big((p_i \N + c_i) \cap H_i \big) - F \supseteq \bigcap_{\ell = 0}^{p_i - 1} (H_i - f_{\ell}),\]
    which is thick, as desired. 
    
    Case 2: For all $i \in \N$, the set $\N \setminus B$ avoids some modulo $p_i$ congruence class. In this case, for all $i \in \N$, there exists $a_i \in \N$ such that $B \supseteq p_i \N + a_i$. We will show that $B$ is thick. Indeed, let $k \in \N$. It suffices to show there exists $n \in \N$ such that $n + i \in p_i \N + a_i$ for all $1 \leq i \leq k$. This is equivalent to showing that there exists $n \in \N$ such that $n \equiv a_i - i \bmod p_i$ for $1 \leq i \leq k$. This is an immediate consequence of the Chinese Remainder Theorem since the $p_i$'s are distinct primes.

    To see that the set $A$ is not an IP set under the stipulated conditions, note that for all $n \in \N$, the set $A \cap (A-n)$ is contained in a finite union of sets of the form $(p \N + c ) \cap H$, where $c \not \equiv 0 \bmod p$ and $H$ is thick.  If $A$ was an IP set, there would exist $n \in \N$ for which $A \cap (A-n)$ is an IP set.  Since the family $\IP$ is partition regular, it would follow that a set of the form $(p \N + c ) \cap H$ is an IP set, which is false as every IP set contains infinitely many multiples of $p$.
\end{proof}

\subsection{Dynamically thick sets from disjointness}
\label{sec_dy_thick_from_disjointness}

In this subsection, we generalize the example in \cref{rem:combinatorial_primes} by showing that every minimal system must be nearly disjoint from almost all systems in a disjoint collection of minimal systems.  This is reminiscent of the fact that in a Hilbert space, according to Bessel's inequality, every vector must be nearly orthogonal to almost all vectors in an orthonormal sequence.

We call a collection of minimal systems $(X_i, T_i)$, $i \in I$, \emph{disjoint} if the product system $(\Pi_i X_i, \Pi_i T_i)$ is minimal.  Note that by the definition of the product topology, this is equivalent to having that the product system of any finite subcollection of the systems is minimal. 

For our purposes below, a \emph{joining} of the systems $(X,T)$ and $(Y,S)$ is a subsystem $(Z \subseteq X \times Y, T \times S)$ of the product system for which $\pi_X Z = X$ and $\pi_Y Z = Y$. Thus, two minimal systems are disjoint if their only joining is the product system.  It is easy to see that if $(X,T)$ and $(Y,S)$ are both minimal, then every subsystem of $(X \times Y, T \times S)$ is a joining.  In this case, every joining of $(X,T)$ and $(Y,S)$ has a minimal subsystem that is a joining.

\begin{theorem}
\label{thm_bessel}
    Let $(Z_i,R_i)$, $i \in \N$, be a disjoint collection of minimal systems.  For all minimal $(X,T)$, all $\eps > 0$, and all sufficiently large $i$ (depending on the $(Z_i,R_i)$'s, $(X,T)$, and $\eps$), all joinings $(J,T \times R_i)$ of $(X,T)$ and $(Z_i,R_i)$ satisfy: for all $z \in Z_i$, the set $\pi_1(J \cap (X \times \{z\}))$ is $\eps$-dense in $X$.
\end{theorem}

\begin{proof}
    Let $(X,T)$ be minimal and $\eps > 0$. Suppose for a contradiction that there are infinitely many $i$'s for which there exists a joining $(J_i, T \times R_i)$ of $(X,T)$ and $(Z_i,R_i)$ and a point $y_i \in Z_i$ for which $\pi_1(J_i \cap (X \times \{y_i\}))$ is not $\eps$ dense in $X$.   To save on notation, by relabeling, we will ignore those $(Z_i,R_i)$'s which do not fall into this infinite set. By passing to minimal subsystems, we may assume without loss of generality that the joinings $(J_i, T \times R_i)$ are minimal. Since $J_i$ is closed, there exists an $\eps$-ball $U_i \subseteq X$ and an open set $V_i \subseteq Z_i$ containing $y_i$ such that $J_i \cap (U_i \times V_i) = \emptyset$.

    Fix $x_0 \in X$.  Since $(J_i,T \times R_i)$ is a joining of $(X,T)$ and $(Z_i,R_i)$ and is minimal, there exists $z_i \in Z_i$ such that $J_i = \overline{(T \times R_i)^{\N}(x_0,z_i)}$. Since $J_i \cap (U_i \times V_i) = \emptyset$, we have that
    \begin{align}
    \label{eqn_main_return-time_set_containment}
        R_{R_i}(z_i,V_i) \subseteq R_T(x_0, X \setminus U_i).
    \end{align}
    Let $c_1, c_2, \ldots, c_m \in X$ be an $\eps/2$ dense subset of $X$.  For $\ell \in \{1, \ldots, m\}$, define
    \[
        E_\ell = \{i \in \N \ \big| \ B_{\eps/2}(c_\ell) \subseteq U_i \}.
    \]
    Note that $\N = \bigcup_{\ell=1}^m E_\ell$. 
    There exists $\ell \in \{1, \ldots, m\}$ for which $E_\ell$ is infinite.  Put $U = B_{\eps/2}(c_\ell)$. It follows now from \eqref{eqn_main_return-time_set_containment} that
    \begin{align}
    \label{eqn_main_return-time_set_containment_two}
        H \coloneqq \bigcup_{i \in E_\ell} R_{R_i}(z_i,V_i) \subseteq  \bigcup_{i \in E_\ell} R_T(x_0, X \setminus U_i) \subseteq R_T(x_0, X \setminus U).
    \end{align}
    We will show that the set $H$ is thickly syndetic and, hence, thick.  This will yield a contradiction since, by the minimality of $(X,T)$, the set $R_T(x_0, X \setminus U)$ cannot be thick.

    Let $k \in \N$, and let $e_1, \ldots, e_k \in E_\ell$. We see that
    \[\bigcap_{i=1}^k (H-i) \supseteq \bigcap_{i=1}^k R_{R_{e_i}}(R_{e_i}^{i}z_{e_i},V_{e_i}) = R_{R_{e_1} \times \cdots \times R_{e_k}}\big((R_{e_1}^1z_{e_1}, \ldots, R_{e_k}^{k}z_{e_k}),V_{e_1} \times \cdots \times V_{e_k}\big),\]
    which is syndetic by the minimality of the product system $(Z_{e_1} \times \cdots \times Z_{e_k},R_{e_1} \times \cdots \times R_{e_k})$. Therefore, intervals of length $k$ appear in $H$ syndetically. Since $k \in \N$ was arbitrary, this shows that the set $H$ is thickly syndetic, as desired.
\end{proof}

\begin{proposition}
\label{thm_quantitative_strengthening_of_type_1_example}
    Let $(X_i,T_i)$, $i \in \N$, be a disjoint collection of minimal systems.  For all minimal systems $(Y,S)$ and all nonempty, open $V \subseteq Y$, there exists $i \in \N$ such that the following holds.  For all nonempty, open $U \subseteq X_i$, there exists $s \in \N$ such that for all $(x,y) \in X_i \times Y$, the set $R_{T_i \times S}((x,y), U \times V) = R_{T_i}(x,U) \cap R_S(y,V)$ is syndetic with gap size bounded by $s$.
\end{proposition}

\begin{proof}
    Let $(Y,S)$ be minimal and $V \subseteq Y$ be nonempty and open.  Let $\eps > 0$ be such that $V$ contains an $\eps$-ball.  According to \cref{thm_bessel}, there exists $i \in \N$ such that all joinings $(J,T_i \times S)$ of $(X_i,T_i)$ and $(Y,S)$ satisfy: for all $x \in X_i$, the set $\pi_2(J \cap (\{x\} \times Y))$ is $\eps$-dense in $Y$.

    Let $U \subseteq X$ be nonempty and open.  We claim that for all $(x,y) \in X_i \times Y$, there exists $n \in \N$ such that $(T_i \times S)^n (x,y) \in U \times V$.  Indeed, let $(x,y) \in X_i \times Y$. By a theorem of Auslander \cite{auslander_1960} and Ellis \cite{ellis_1957}, the point $(x,y)$ is proximal to a point $(x_0,y_0) \in X_i \times Y$ that is uniformly recurrent under $T_i \times S$.  Since $J \defeq \overline{(T_i \times S)^{\N_0}(x_0,y_0)}$ is a (minimal) joining of $(X_i,T_i)$ and $(Y,S)$, it is $\eps$-dense in every fiber over $X_i$.  Since $V$ contains an $\eps$-ball, we see that $J \cap (U \times V) \neq \emptyset$.  Because $(x,y)$ is proximal to $(x_0,y_0)$ and $R_{T_i \times S}((x_0,y_0), U \times V)$ is syndetic, we see that $R_{T_i \times S}((x,y),U \times V)$ is piecewise syndetic, hence nonempty.
    
    We have shown that for all $(x,y) \in X_i \times Y$, there exists $n \in \N$ such that $(T_i \times S)^n (x,y) \in U \times V$. It follows that $\bigcup_{n=1}^\infty (T_i \times S)^{-n} (U \times V) = X_i \times Y$.  By compactness, there exists $s \in \N$ such that $\bigcup_{n=1}^s (T_i \times S)^{-n} (U \times V) = X_i \times Y$. It follows that for all $(x,y) \in X_i \times Y$, there exists $n \in \{1, \ldots, s\}$ such that $(T_i \times S)^n(x,y) \in U \times V$, whereby for all $(x,y) \in X_i \times Y$, the set $R_{T_i \times S}((x,y), U\times V)$ is syndetic with gap size bounded by $s$, as was to be shown.
\end{proof}

Taking $(X_i, T_i)$ to be a rotation on $p_i$-many points, the following theorem generalizes \cref{rem:combinatorial_primes}.

\begin{theorem}
\label{lemma_more_general_example_of_dy_thick_not_IP}
    Let $(X_i,T_i)$, $i \in \N$, be a disjoint collection of minimal systems.  For each $i \in \N$, let $x_i \in X_i$, $U_i \subseteq X_i$ be nonempty and open, and $H_i \subseteq \N$ be thick.   The set
    \[
        A \defeq \bigcup_{i=1}^\infty \big( R_{T_i}(x_i, U_i) \cap H_i \big)
    \]
    is dynamically thick.
\end{theorem}

\begin{proof}
    To see that $A \in \dthick$, we must show that for all minimal $(Y,S)$, all nonempty, open $V \subseteq Y$, and all $y \in Y$, the set $A \cap R_S(y,V)$ is nonempty.  This follows immediately from \cref{thm_quantitative_strengthening_of_type_1_example}.
\end{proof}

\subsection{Dynamically thick sets from distal points}
\label{sec_dy_thick_from_distal}

In this subsection, we offer a simultaneous generalization of the examples in \cref{ex:type2-special} and \cref{rem:combinatorial_primes} by making use of distal points (recall the definition from \cref{sec_top_dynamics}).

\begin{lemma}
\label{lemma_orbit_along_ps_is_somewhere_dense}
    Let $(X,T)$ be a minimal system and $x \in X$.  For all piecewise syndetic sets $P \subseteq \N$, the set $T^P x = \{T^n x: n \in P\}$ is somewhere dense, ie., $(\overline{T^P x})^\circ \neq \emptyset$.
\end{lemma}

\begin{proof}
    Since the set $P$ is piecewise syndetic, there exists $k \in \N$ such that $H \defeq \bigcup_{i=1}^k (P-i)$ is thick.  Since $(X,T)$ is minimal, the set $T^H x$ is dense.  We see that
    \[X = \overline{\bigcup_{i=1}^k T^{-i} T^{P}x} = \bigcup_{i=1}^k T^{-i} \overline{T^{P}x}.\]
    By the Baire Category Theorem, there exits $i \in \{1, \ldots, k\}$ such that the set $T^{-i} \overline{T^{P}x}$ has nonempty interior.  By semiopenness of minimal maps \cite[Thm 2.4]{kolyada_snoha_trofimchuk_2001}, applying $i$ many times the map $T$, the set $\overline{T^{P}x}$ has nonempty interior, as desired.
\end{proof}

By considering a periodic system and an infinite product of periodic systems of prime cardinality, the following theorem generalizes the examples in both \cref{ex:type2-special} and \cref{rem:combinatorial_primes}, respectively.

\begin{theorem}
\label{prop_description_of_dy_thick}
    Let $(X, T)$ be a minimal system and $U_1, U_2, \ldots$ be nonempty, open subsets of $X$. Let $x \in X$ be a distal point and $H_1, H_2, \ldots \subseteq \N$ be thick sets. The set
    \[
        A \defeq \bigcup_i \big(R(x,U_i) \cap H_i \big)
    \]
    is dynamically thick if and only if the set $\bigcup_{i=1}^\infty U_i$ is dense in $X$.
\end{theorem}

\begin{proof}
    If the set $\bigcup_{i=1}^\infty U_i$ is not dense in $X$, then it is disjoint from a nonempty, open set $V \subseteq X$. The set $A$ is clearly disjoint from the set $R(x,V)$, whereby $A$ is not dynamically thick.

    Suppose that $\bigcup_{i=1}^\infty U_i$ is dense. Let $(Y,S)$ be a minimal system, $y \in Y$, and $V \subseteq Y$ be nonempty and open.  Since the set $R_S(y,V)$ is syndetic, by \cref{lemma_orbit_along_ps_is_somewhere_dense}, the set $T^{R_S(y,V)} x$ is somewhere dense. Therefore, there exists $i \in \N$ such that $T^{R_S(y,V)} x \cap U_i \neq \emptyset$.  It follows that the set $R_{T \times S} ((x,y), U_i \times V)$ is nonempty. Since $x$ is a distal point, by \cite[Thm. 9.11]{furstenberg_book_1981}, the point $(x,y)$ is uniformly recurrent under $T \times S$.  Since $U_i \times V$ is open, we have that $R_{T \times S} ((x,y), U_i \times V)$ is (dynamically) syndetic. Therefore, the set $R_{T \times S} ((x,y), U_i \times V) \cap H_i$ is nonempty, which implies that $A \cap R_S(y,V)$ is nonempty, as desired.
\end{proof}

\subsection{A general form for dynamically thick sets: a proof of \texorpdfstring{\cref{mainthm_structure_for_dy_thick_sets}}{Theorem D}}

\label{sec_general_form_of_dthick}

In this subsection, we draw on a future result, \cref{thm:relations_C_dPS_PS-intro}, to describe the form of any dynamically thick set.  (For the astute reader concerned about circular logic: while \cref{cor_weak_structure_result} relies on \cref{thm:relations_C_dPS_PS-intro}, we do not invoke \cref{cor_weak_structure_result} anywhere in the paper.) We recall the definition of robustly syndetic collections from the introduction.

\begin{definition}
\label{def_robust_syndetic}
A collection $\collection$ of subsets of $\N$ is \emph{robustly syndetic} if for all dynamically syndetic sets $A \subseteq \N$, there exists $B \in \collection$ such that $A \cap B$ is syndetic.
\end{definition}

The dynamically thick sets described so far all take the form $\bigcup_{B \in \collection} (B \cap H_B)$, where $\collection$ is a robustly syndetic collection of sets and $(H_B)_{B \in \collection}$ is a collection of thick sets.  Indeed, in \cref{lemma_more_general_example_of_dy_thick_not_IP}, the collection $\{R_{T_i}(x_i,U_i) \ | \ i \in \N\}$ is robustly syndetic by virtue of the disjointness of the collection of systems $(X_i, T_i)$ for $i \in \N$.  In \cref{prop_description_of_dy_thick}, the collection $\{R_{T}(x,U_i) \ | \ i \in \N\}$ is robustly syndetic by virtue of the fact that $x$ is a distal point and $\bigcup_{i=1}^\infty U_i$ is dense in $X$.  \cref{mainthm_structure_for_dy_thick_sets} shows that every dynamically thick set has an underlying robustly syndetic collection.

\begin{proof}[Proof of \cref{mainthm_structure_for_dy_thick_sets}]
    Suppose that $\collection$ is robustly syndetic and that, for each $B \in \collection$, the set $H_B \subseteq \N$ is thick such that \eqref{eqn_A_contains_family_cap_thick} holds. To see that $A \in \dthick$, let $C \in \dsyndetic$.  There exists $B \in \collection$ such that $B \cap C$ is syndetic. Thus, the set $B \cap C \cap H_B$ is nonempty, whereby $A \cap C \neq \emptyset$, as desired.

    Conversely, suppose that $A$ is dynamically thick.  For all $S \in \dsyndetic$, we will define a thick set $G_S \subseteq \N$ and a set $B_S \subseteq \N$ in such a way that the collection $\collection \defeq \{ B_S \ | \ S \in \dsyndetic\}$ is robustly syndetic and
    \[A \supseteq \bigcup_{S \in \dsyndetic} \big( B_S \cap G_S \big).\]
    This suffices to reach the conclusion: the collection $\collection \cup \{A\}$ is robustly syndetic, and setting $H_A \defeq \N$ and, for $B = B_S \in \collection$, setting $H_B \defeq G_{S}$, we see that equality in \eqref{eqn_A_contains_family_cap_thick} holds.

    Let $S \in \dsyndetic$. The set $A \cap S$ is an intersection of a dynamically thick set and a dynamically syndetic set and, hence, by \cref{thm:relations_C_dPS_PS-intro}, is piecewise syndetic.  There exist $\ell \in \N$ and finite, disjoint intervals $I_1$, $I_2$, \dots of consecutive positive integers satisfying: a) $\lim_{i \to \infty} |I_i| \to \infty$ and, for all $i \in \N$, b) $\max I_i  + i< \min I_{i+1}$; and c) the set $A \cap S$ has nonempty intersection with all subintervals of $I_i$ of length at least $\ell$.  
    
    Define $G_S \defeq \bigcup_{i=1}^\infty I_i$, and note that it is thick.  Define
    \begin{align}
    \label{eq:B_S}
        B_S \defeq (A \cap G_S) \cup (\N \setminus G_S).
    \end{align}
    Since the set $\N \setminus G_S$ is a disjoint union of finite intervals whose lengths tend to infinity, the set
    \begin{align}
    \label{eq:B_ScapS}
        B_S \cap S = (A \cap G_S \cap S) \cup \big((\N \setminus G_S) \cap S \big) \text{ is syndetic.}
    \end{align}
    Indeed, the set $S$ has nonempty intersection with any interval of length $\ell' \in \N$.  Therefore, on $G_S$, the set $B_S \cap S$ has nonempty intersection with any interval of length at least $\ell$, while on $\N \setminus G_S$, the set $B_S \cap S$ has nonempty intersection with any interval of length $\ell'$.
    
    Now \eqref{eq:B_ScapS} shows that the collection $\collection \defeq \{B_S \ | \ S \in \dsyndetic\}$ is robustly syndetic, and \eqref{eq:B_S} shows that $A \supseteq \bigcup_{S \in \dsyndetic} (B_S \cap G_S)$, as desired.
\end{proof}

It would be interesting and useful to improve \cref{cor_weak_structure_result} by saying more about the robustly syndetic collection $\collection$.  We discuss this further in \cref{quest_robust_synd_family}.

\section{The splitting problem and \texorpdfstring{$\sigma$}{sigma}-compactness}
\label{sec:sigma-compact}

The motivating question for this section is: can a set of pointwise recurrence be partitioned into two disjoint sets of pointwise recurrence?  We frame this question as one concerning compactness, beginning with definitions and examples in \cref{sec:compact-sigma-compact-first}.  We show in \cref{sec_partition_in_dual_compact_families} that this framework explains splitting results for a number of well-known families.  Finally,  we prove \cref{theorem_dy_syndetic_is_not_sigma_compact} in \cref{sec_dcsyndetic_not_sigma_compact}, showing that our motivating question cannot be answered in this way.

\subsection{Compact and \texorpdfstring{$\sigma$}{sigma}-compact families}
\label{sec:compact-sigma-compact-first}

The following definition is shown to be equivalent to the one given in the introduction in \cref{prop:sigma-compact-equivalent}.

\begin{definition}
\label{def:compact}
\label{def:uniformity}
Let $\familytwo$ be a family of subsets of $\N$.
\begin{enumerate}
    \item The family $\familytwo$ is \emph{compact} if for all $B \in \familytwo^*$, there exists $M \in \N$ such that $B \cap [M] \in \familytwo^*$.

    \item The family $\familytwo$ is \emph{$\sigma$-compact} if it is the union of countably many compact families.
\end{enumerate}
\end{definition}

Before giving some examples of compact and $\sigma$-compact families, we move to explain the terminology.  For a family $\familytwo$ of subsets of $\N$, define
\[
    X_{\familytwo} \defeq \big\{1_A \in \{0, 1\}^{\N} \ \big| \ A \in \familytwo\big\}.
\]
Thus, the set $X_{\familytwo}$ is a subset of the compact metric space $\{0, 1\}^{\N}$.

\begin{proposition}
\label{prop:sigma-compact-equivalent}
    A family $\familytwo$ is compact (resp. $\sigma$-compact) if and only if the set $X_{\familytwo}$ is compact (resp. $\sigma$-compact).
\end{proposition}

\begin{proof}
Since $X_{\bigcup_{N=1}^\infty \familytwo_N} = \bigcup_{N = 1}^\infty X_{\familytwo_N}$, it suffices to show that $\familytwo$ is compact if and only if $X_{\familytwo}$ is compact.

Suppose $X_{\familytwo}$ is a compact subset of $\{0, 1\}^{\N}$. Assume for the sake of a contradiction that $\familytwo$ is not a compact family. Thus, there exists $B \in \familytwo^*$ such that for every $M \in \N$, there exists $A_M \in \familytwo$ such that
\begin{align}
\label{eqn_f_not_compact}
    B \cap [M] \cap A_M = \emptyset.
\end{align}

Since $X_{\familytwo}$ is compact, the sequence $(1_{A_M})_{M \in \N} \subseteq X_{\familytwo}$ has a subsequence that converges to some point $1_A \in X_{\familytwo}$. By the product topology on $\{0,1\}^{\N}$ and \eqref{eqn_f_not_compact}, we see that $B \cap A = \emptyset$, in contradiction with the fact that $A \in \familytwo$ and $B \in \familytwo^*$. 

Conversely, suppose that $\familytwo$ is a compact family. We will show that $X_{\familytwo}$ is compact by showing that it is closed. Let $(x_k)_{k \in \N} \subseteq X_{\familytwo}$ be a convergent sequence with limit point $x \in \{0,1\}^{\N}$. We will show that $x \in X_{\familytwo}$ by showing that for all $B \in \familytwo^*$, the set $\supp(x) \cap B$ is nonempty.

Let $B \in \familytwo^*$.  Since $\familytwo$ is compact, there exists $M \in \N$ such that for all $k \in \N$,
\[
    B \cap [M] \cap \supp(x_k) \neq \emptyset.
\]
Since $\lim_{k \to \infty}x_k = x$ and $B \cap [M]$ is a finite set, we see that $B \cap [M] \cap \supp(x) \neq \emptyset$. In particular, the set $B \cap \supp(x)$ is nonempty, as was to be shown.
\end{proof}

Next we provide some examples of compact and $\sigma$-compact families.

\begin{example}
\label{ex_syndetic_is_sigma_compact}
    For $N \in \N$, denote by $\syndetic_N$ the subfamily of syndetic sets that have nonempty intersection with every interval in $\N$ of length $N$.  The family $\syndetic_N$ is compact.  Indeed, given $B \in \syndetic_N^*$, the set $B$ is thick.  If $M \in \N$ is large so that $B \cap [M]$ contains an interval of length at least $N$, then $B \cap [M] \in \syndetic_N^*$.  We see that the family of syndetic sets, $\syndetic$, is equal to $\bigcup_{N=1}^\infty \syndetic_N$, whereby it is $\sigma$-compact.
\end{example}

\begin{example}
\label{ex_ip_k_sets_are_compact}
    Let $N \in \N$. Denote by $\IP_N$ the upward closure of the collection of sets of the form
    \[\FS(x_i)_{i=1}^N \defeq \left\{ \sum_{f \in F} x_f \ \middle | \ \emptyset \neq F \subseteq \{1, \ldots, N\} \right\}, \quad x_1, \ldots, x_N \in \N.\]
    As in the previous example, the family $\IP_N^*$ is easily seen from the definition to be compact.  The dual of the family $\IP_0 \defeq \bigcap_{N = 1}^\infty \IP_N$ is $\sigma$-compact, since $\IP_0^* = \bigcup_{N=1}^\infty \IP_N^*$.  The family $\IP_0^*$ is important in quantitative strengthenings of recurrence theorems and appears again in this paper in \cref{sec_open_quests_dps_sets}.
\end{example}

\begin{example}
\label{example:set-topological-partition}
    The family of sets of return times in minimal systems is $\sigma$-compact. More precisely, let $\familytwo$ be the upward closure of collection of sets of the form
    \begin{equation}\label{eq:sigma_compact_topological}
        \big\{n \in \N \ \big| \ U \cap T^{-n} U \neq \emptyset \big\},
    \end{equation}
    where $(X, T)$ is a minimal system and $U \subseteq X$ is a nonempty, open set. The family $\familytwo$ is $\sigma$-compact, as shown in \cite[Prop. 5.8]{Boshernitzan-Glasner-tworecurrence}. The argument is short, so we give it here.
    
    It is a fact \cite[Thm. 2.4]{bergelson_mccutcheon_1998} that $A \in \familytwo$ if and only if there exists a syndetic set $S \subseteq \N$ such that $A \supseteq S-S$. Thus, using the notation from \eqref{ex_syndetic_is_sigma_compact}, if for $N \in \N$ we define
    \[
        \familytwo_N \defeq \upclose \big\{ S - S \ \big| \ S \in \syndetic_N \big\},
    \]
    we have that $\familytwo = \bigcup_{N=1}^{\infty} \familytwo_N$.  We have only to show that for all $N \in \N$, the family $\familytwo_N$ is compact. Let $N \in \N$ and $B \in \familytwo_N^*$.  We must show that there exists $M \in \N$ such that for all $A \in \familytwo_N$, the set $A \cap B \cap [M]$ is nonempty.

    Assume for the sake of a contradiction that no such $M$ exists.  Thus, for all $M \in \N$, there exists $S_M \in \syndetic_N$ such that
    \begin{align}
    \label{eqn_sm_limit}
        (S_M - S_M) \cap B \cap [M] = \emptyset.
    \end{align}
    Since $\syndetic_N$ is compact, passing to a subsequence of $(S_M)_{M \in \N}$, there exists $S \in \syndetic_N$ such that $\lim_{M \to \infty} 1_{S_M} = 1_S$.  Since $S \in \syndetic_N$, the set $S-S$ is a member of $\familytwo_N$.  It follows from \eqref{eqn_sm_limit} that $(S-S) \cap B = \emptyset$, contradicting the fact that $B \in \familytwo_N^*$.
\end{example}

\begin{example}
\label{example:set-measurable-partition}
    In analogy to \cref{example:set-topological-partition}, the family of sets of return times in measure preserving systems is $\sigma$-compact. More precisely, the upward closure $\familytwo$ of the collection of sets of the form
    \begin{equation}\label{eq:sigma_compact_measurable}
        R(E,E) \defeq \big\{n \in \N \ \big| \ \mu(E \cap T^{-n} E)  > 0 \big\}, 
    \end{equation}
    where $(X, \mu, T)$ is a measure preserving system (see \cref{sec_dcps_and_set_recurrence}) and $E \subseteq X$ is a set of positive measure, is $\sigma$-compact.  This is most readily seen as a consequence of a result of Forrest \cite{Forrest_recurrence-in-dynamical-systems}: for all sets $B \subseteq \N$ of measurable recurrence (i.e. for all $B \in \familytwo^*$) and all $\delta > 0$, there exists $M \in \N$ such that for all sets $E$ with measure at least $\delta$, the set $B \cap R(E,E) \cap [M]$ is nonempty. For $N \in \N$, defining
    \[
        \familytwo_N \defeq \upclose \big\{R(E,E) \subseteq \N \ \big| \ \text{$(X,\mu,T)$ is a measure preserving system, $E \subseteq X$, $\mu(E) > 1/N$} \big\},
    \]
    we see that $\familytwo_N$ is compact and $\familytwo = \bigcup_{N =1}^\infty \familytwo_N$ is $\sigma$-compact.
\end{example}

While the previous two examples concern set recurrence, the next one is about pointwise recurrence.
\begin{example}
    Bohr$_0$ sets and  Nil$_k$-Bohr$_0$ sets are special dynamically central syndetic sets for which the systems in consideration are minimal rotations on compact abelian groups and minimal $k$-step nilsystems, respectively. (See \cite{Host-Kra_nilbohr} for a detailed definition of Nil$_k$-Bohr$_0$ sets.)
    It can be shown that these families are $\sigma$-compact using the results in \cite[Prop. 1.4]{Le-interpolation-first} and \cite[Lemma 3.3]{Le_interpolation_for_nilsequences}, respectively.
\end{example}

\subsection{The splitting problem in dual compact families}
\label{sec_partition_in_dual_compact_families}

The following result offers, under a mild assumption on the family $\family \classcap \family^*$, a solution to the splitting problem in families whose duals are $\sigma$-compact.

\begin{proposition}
\label{prop:sigma-compact-implies-partition-to-two}
    Let $\family$ be a family of subsets of $\N$ such that that all members of $\family \classcap \family^*$ are infinite.  If $\family^*$ is $\sigma$-compact, then for all $A \in \family$, there exists a disjoint partition $A = A_1 \cup A_2$ with $A_1, A_2 \in \family$.
\end{proposition}

\begin{proof}
    Suppose $\family^*$ is $\sigma$-compact. There exists $\family^* = \bigcup_{N=1}^{\infty} \familytwo_N$ where each family $\familytwo_N$ is compact. By taking finite unions, we may assume without loss of generality that $\familytwo_1 \subseteq \familytwo_2 \subseteq \cdots$.
    
    Let $A \in \family$. Since $\familytwo_1 \subseteq \family^*$, the set $A$ belongs to $\familytwo_1^*$. Since $\familytwo_1$ is compact, there exists a finite set $A_1' \subseteq A$ such that $A_1' \in \familytwo_1^*$. Since all members of $\family \classcap \family^*$ are infinite, for all $B \in \family^*$, the set $A \cap B$ is infinite, whereby
    \[
        \big(A \setminus A_1' \big) \cap B \neq \emptyset. 
    \]
    Since $B \in \family^*$ was arbitrary, we see that $A \setminus A_1' \in \family$. Repeating the argument for $A \setminus A_1'$ in place of $A$, we get a finite set $A_2' \subseteq A \setminus A_1'$ such that $A_2' \in \familytwo_2^*$. By the same argument as before, the set $A \setminus (A_1' \cup A_2')$ is a member of $\family$, allowing us to continue on as before.
    
    Repeating ad infinitum, we get a sequence $A_1'$, $A_2'$, \dots of disjoint, finite subsets of $A$ such that for all $N \in \N$, the set $A_N' \in \familytwo_N^*$.  Define $A_1 \defeq A_1' \cup A_3' \cup \cdots$ and $A_2 \defeq A_2' \cup A_4' \cup \cdots$.  By appending any elements in $A \setminus (A_1 \cup A_2)$ to $A_1$, we may assume that $A = A_1 \cup A_2$.
    
    We need only to show that $A_1$ and $A_2$ are members of $\family$.  Indeed, since $\family^* = \bigcup_{N=1}^{\infty} \familytwo_N$, we see that $\family = (\family^*)^* = \bigcap_{N=1}^{\infty} \familytwo_N^*$.  Since $\familytwo_1^* \supseteq \familytwo_2^* \supseteq \cdots$, by the construction of the $A_i$'s, both $A_1$ and $A_2$ belong to $\bigcap_{N=1}^{\infty} \familytwo_N^*$.
\end{proof}

\begin{remark}
    In view of \cref{prop:sigma-compact-implies-partition-to-two} and the examples of $\sigma$-compact families presented in \cref{sec:compact-sigma-compact-first}, every set in the following  families can be partitioned into two sets in the same family:
    \begin{enumerate}
        \item sets of topological recurrence (dual of sets of the form in \eqref{eq:sigma_compact_topological}),
        \item sets of measurable recurrence (dual of sets of the form in \eqref{eq:sigma_compact_measurable}), 
        \item sets of Bohr recurrence (dual of Bohr$_0$ sets),
        \item sets of pointwise recurrence for $k$-step nilsystems (dual of Nil$_k$-Bohr$_0$ sets).
    \end{enumerate}
    These facts are well known to experts, but the framing in terms of compactness is, as far as we know, new.
\end{remark}

\subsection{Proof of \texorpdfstring{\cref{theorem_dy_syndetic_is_not_sigma_compact}}{Theorem E}}
\label{sec_dcsyndetic_not_sigma_compact}

We show in this section that the families $\dsyndetic$ and $\dcsyndetic$ are not $\sigma$-compact.

\begin{lemma}
\label{thm_suff_condition_for_dual_not_sigma_compact}
    Let $\family$ be a family of subsets of $\N$.  If
    \begin{align}
        \label{eqn_key_property_to_refute_sigma_compactness_one}
        \begin{gathered}
            \text{there exist $B_1, B_2, \ldots \in \family$, disjoint, such that for all $B \in \family$} \\ \text{with $B \subseteq \bigcup_{i=1}^\infty B_i$, there exists $i \in \N$ such that $|B \cap B_i| = \infty$,}
        \end{gathered}
    \end{align}
    then the family $\family^*$ is not $\sigma$-compact.
\end{lemma}

\begin{proof}
    Suppose property \eqref{eqn_key_property_to_refute_sigma_compactness_one} holds, and suppose for a contradiction that $\family^* = \bigcup_{i=1}^\infty \familytwo_i$, where $\familytwo_1$, $\familytwo_2$, \dots are compact families.

    Let $i \in \N$.  We claim that there exists a function $\varphi_i: \class \to \N$ such that for all $A \in \classtwo_i$ and for all $B \in \class$,
    \[
        \min (A \cap B) \leq \varphi_i(B).
    \]
    Indeed, fix $B \in \class$, and consider the function $X_{\classtwo_i} \to \N$ defined by $1_A \mapsto \min (A \cap B)$.  It is locally constant since, more generally, the function $X_{\class} \times X_{\class^*} \to \N$ defined by $(1_A,1_B) \mapsto \min (A \cap B)$ is locally constant. Since $\classtwo_i$ is compact, the function $X_{\classtwo_i} \to \N$ is bounded from above.  We define $\varphi_i(B)$ to be an upper bound.

    Define the set
    \[
        A \defeq \left( \N \setminus \bigcup_{i=1}^\infty B_i \right) \cup \bigcup_{i=1}^\infty \big( B_i \setminus \{1, \ldots, \varphi_i(B_i) \} \big).
    \]
    We claim that $A \in \class^*$.  Indeed, if $B \in \class$, then either $B \cap ( \N \setminus \bigcup_{i=1}^\infty B_i ) \neq \emptyset$, in which case $A \cap B \neq \emptyset$, or $B \subseteq \bigcup_{i=1}^\infty B_i$.  In the latter case, by \eqref{eqn_key_property_to_refute_sigma_compactness_one}, there exists $i \in \N$ such that $|B \cap B_i| = \infty$.  It follows that $B \cap (B_i \setminus \{1, \ldots, \varphi_i(B_i) \}) \neq \emptyset$, whereby $A \cap B \neq \emptyset$.

    Since $A \in \class^*$, there exists $j \in \N$ such that $A \in \classtwo_j$.  It follows that $\min (A \cap B_j) \leq \varphi_j(B_j)$. But since the $B_j$'s are disjoint, we see from the definition of $A$ that
    \[
        A \cap B_j = B_j \setminus \{1, \ldots, \varphi_j(B_j) \}
    \]
    whereby $\min (A \cap B_j) > \varphi_j(B_j)$, a contradiction.
\end{proof}

The following lemma gives a sufficient condition to satisfy the property in \eqref{eqn_key_property_to_refute_sigma_compactness_one}.  Recall that all difference sets are computed in $\N$.

\begin{lemma}
\label{lemma_suff_condition_for_key_property_for_dual_not_sigma_compact}
    Let $\class$ be a family of subsets of $\N$.  If
    \begin{align}
        \label{eqn_suff_condition_for_key_prop_one}
        \text{for all $B \in \class$, there exists $n \in \N$ such that $|B \cap (B-n)| = \infty$,}
    \end{align}
    and if
        \begin{align}
        \label{eqn_suff_condition_for_key_prop_two}
        \begin{gathered}
            \text{there exist $B_1, B_2, \ldots \in \class$, disjoint, such that} \\ \lim_{d \to \infty} \min \bigcup_{\substack{i, j = 1 \\ \max(i,j) > d}}^\infty (B_i - B_j) = \infty,
        \end{gathered}
    \end{align}
    then the property in \eqref{eqn_key_property_to_refute_sigma_compactness_one} holds.
\end{lemma}

\begin{proof}
    Suppose the properties in \eqref{eqn_suff_condition_for_key_prop_one} and \eqref{eqn_suff_condition_for_key_prop_two} hold.  We will show that the property in \eqref{eqn_key_property_to_refute_sigma_compactness_one} holds for the same sets $B_1$, $B_2$, \dots.

    Let $B \in \class$ be such that $B \subseteq \bigcup_ {i=1}^\infty B_i$.  By \eqref{eqn_suff_condition_for_key_prop_one}, there exists $n \in \N$ such that $|B \cap (B-n)| = \infty$.  Writing $B = B \cap \bigcup_{i=1}^\infty B_i$, we see that
    \[B \cap (B-n) = \bigcup_{i,j = 1}^\infty \big( B \cap B_i \cap (B-n) \cap (B_j - n)\big).\]
    By \eqref{eqn_suff_condition_for_key_prop_two}, there exists $d \in \N$ such that if $\max(i,j) > d$, then $B_i \cap (B_j - n) = \emptyset$.  Therefore,
        \[B \cap (B-n) = \bigcup_{i,j = 1}^d \big( B \cap B_i \cap (B-n) \cap (B_j - n)\big) \subseteq \bigcup_{i=1}^d \big( B \cap B_i \big).\]
    By assumption, the set on the left hand side is infinite, so there must exist an $i \in \{1, \ldots, d\}$ such that $B \cap B_i$ is infinite.  This verifies \eqref{eqn_key_property_to_refute_sigma_compactness_one}.
\end{proof}

The result in the following lemma is well known.  We include a proof for completeness.

\begin{lemma}
\label{lem:piecewise_syndetic_shift}
    If $B \subseteq \N$ is piecewise syndetic, then there exists $n \in \N$ such that the set $B \cap (B - n)$ is piecewise syndetic.
\end{lemma}

\begin{proof}
    Suppose $B \subseteq \N$ is piecewise syndetic. There exists a syndetic set $S \subseteq \N$ and a thick set $H \subseteq \N$ such that $B = S \cap H$.  There exists $N \in \N$ such that $\bigcup_{i=1}^N (S-i) \supseteq \N$, and so $\bigcup_{i=1}^N (B-i) \supseteq \bigcap_{i=1}^N (H-i)$. It follows that the set 
    \[
        B \cap \bigcup_{i=1}^N (B-i) \supseteq S \cap H \cap \bigcap_{i=1}^N (H-i)
    \]
    is piecewise syndetic.  Since the family $\PS$ is partition regular, there exists $n \in \{1, \ldots, N\}$ such that the set $B \cap (B-n)$ is piecewise syndetic.
\end{proof}

\begin{proof}[Proof of \cref{theorem_dy_syndetic_is_not_sigma_compact}]
    We will show that $\dthick$ and $\dcthick$ satisfy the conditions in \cref{lemma_suff_condition_for_key_property_for_dual_not_sigma_compact} and hence in \eqref{eqn_key_property_to_refute_sigma_compactness_one}. It will follow by \cref{thm_suff_condition_for_dual_not_sigma_compact} that the families $\dthick^* = \dsyndetic$ and $\dcthick^* = \dcsyndetic$ are not $\sigma$-compact.
    
    To see that \eqref{eqn_suff_condition_for_key_prop_one} holds for $\dthick$ and $\dcthick$, let $B \in \dthick$ or $B \in \dcthick$. By \cref{thm_dy_c_thick_is_ps}, the set $B$ is piecewise syndetic. Thus, by \cref{lem:piecewise_syndetic_shift}, there exists $n \in \N$ for which the set $B \cap (B-n)$ is piecewise syndetic and hence infinite.

    We will show that property \eqref{eqn_suff_condition_for_key_prop_two} in \cref{lemma_suff_condition_for_key_property_for_dual_not_sigma_compact} holds for the family $\dthick$. Since $\dthick \subseteq \dcthick$, it will follow that $\dcthick$ also satisfies \eqref{eqn_suff_condition_for_key_prop_two}. Let $\{T_{i,j} \ | \ i, j \in \N\}$ be a doubly-indexed family of thick subsets of $\N$ that are well separated, in the sense that
    \begin{align}
    \label{eqn_very_well_separated_thick_sets}
        \lim_{d \to \infty} \min \bigcup_{\substack{ i, j, k, \ell \in \N \\ (i,j) \neq (k, \ell) \\ \max(i,j,k,\ell) > d}} \min \big(T_{i,j} - T_{k,\ell} \big) = \infty,
    \end{align}
    where the notation $(i,j)$ denotes an element of $\N^2$.
    Let $p_1, p_2, \ldots$ be an increasing enumeration of the primes.  For $i \in \N$, define
    \[B_i = \bigcup_{j = i}^\infty \big(T_{i,j} \cap (p_j \N + 1) \big).\]
    The sets $B_1$, $B_2$, \dots are disjoint, and by \cref{rem:combinatorial_primes}, each is dynamically thick.  To verify the limit in \eqref{eqn_suff_condition_for_key_prop_two}, note that for all $i, k \in \N$,
    \[B_i - B_k = \bigcup_{\substack{j=i\\ \ell = k}}^\infty \Big ( \big( T_{i,j} \cap (p_j \N + 1) \big) - \big( T_{k,\ell} \cap (p_\ell \N + 1) \big) \Big) \subseteq \bigcup_{\substack{j=i\\ \ell = k}}^\infty \big( T_{i,j} - T_{k,\ell} \big).\]
    For all $d \in \N$, considering the cases $i=k$ and $i \neq k$ separately, we see that
    \[\min \bigcup_{\substack{i,k = 1 \\ \max(i,k) > d}}^\infty \big( B_i - B_k\big) \geq \min \bigg( \bigcup_{\substack{ i, j, k, \ell \in \N \\ (i,j) \neq (k, \ell) \\ \max(i,j,k,\ell) > d}} \min \big(T_{i,j} - T_{k,\ell} \big), p_{d} \bigg).\]
    The limit in \eqref{eqn_suff_condition_for_key_prop_two} follows now from the limit in \eqref{eqn_very_well_separated_thick_sets} and the fact that $\lim_{d \to \infty} p_d = \infty$, as desired.
\end{proof}

\begin{remark}
Using \cref{lemma_suff_condition_for_key_property_for_dual_not_sigma_compact}, one can show that the families 
\[
    \IP^*, \IP_0^*, \central^*, \PS^*, \dthick, \dcthick, \thick, \IP, \IP_0, \central, \PS 
\]
are not $\sigma$-compact. However, members of the duals of the first seven families $\IP = (\IP^*)^*$, $\IP_0 = (\IP_0^*)^*$, $\central = (\central^*)^*$, $\PS = (\PS^*)^*$, $\dsyndetic = \dthick^*$, $\dcsyndetic = \dcthick^*$, $\syndetic = \thick^*$ can be partitioned into two sets belonging to the same family. This demonstrates that $\sigma$-compactness of $\family^*$ is sufficient, but not necessary, for every member of $\family$ to be partitioned into two members of $\family$.
\end{remark}

\section{Dynamically piecewise syndetic sets}
\label{sec_dps_sets}

In this section, we introduce and study the family of dynamically piecewise syndetic sets, $\dPS$, and the family of dynamically central piecewise syndetic sets, $\dcPS$; recall their definitions from the top of \cref{sec_dy_pws_sets}. The main result, \cref{thm:relations_C_dPS_PS-intro}, describes the relationships between these families and other established families.
We deduce from this a number of properties of dynamically piecewise syndetic sets in \cref{sec_dcps_and_set_recurrence}.

\subsection{Members of syndetic, idempotent filters are central}

Before treating dynamically piecewise syndetic sets, we will show that members of syndetic, idempotent filters are central along any thick set.  This fact -- which we will have immediate need for in the next subsection -- is a consequence of \cref{intro_mainthm_full_chars_of_dcsyndetic}, but the argument we give here is far more elementary.

In what follows, for a family $\family$ on $\N$, we denote by $\overline{\family}$ the set $\bigcap_{A \in \family} \overline{A}$ in $\beta \N$.

\begin{lemma}
\label{lemma_syndetic_contained_in_minimal}
    For all syndetic filters $\filter$ on $\N$ and all minimal left ideals $L \subseteq \beta \N$, we have $L \cap \overline{\filter} \neq \emptyset$.
\end{lemma}

\begin{proof}
    We will use the following well-known fact (\cite[Thm. 4.48]{hindman_strauss_book_2012}): for all syndetic sets $A \subseteq \N$ and all minimal left ideals $L \subseteq \beta \N$, we have $\overline{A} \cap L \neq \emptyset$.

    Let $\filter$ be a syndetic filter on $\N$, and let $L \subseteq \beta \N$ be a minimal left ideal.  Consider the set $\{\overline{A} \cap L \ | \ A \in \filter\}$.  By the fact from the previous paragraph, the set consists of closed, nonempty subsets of $L$.  Since $\filter$ is a filter, the set has the finite intersection property.  Since $L$ is compact, we have that $\overline{\filter} \cap L$ is nonempty, as desired.
\end{proof}

\begin{theorem}
\label{thm_sifs_are_contained_in_idempotents}
    For all syndetic, idempotent filters $\filter$ on $\N$ and all minimal left ideals $L \subseteq \beta \N$, there exists an idempotent ultrafilter in $L$ containing $\filter$.
\end{theorem}

\begin{proof}
    Let $\filter$ be a syndetic, idempotent filter on $\N$. By \cref{lemma_syndetic_contained_in_minimal}, there exists a minimal ultrafilter $p \in \beta \N$ such that $\filter \subseteq p$.  Define $L = \beta \N + p$.  Since $p$ is minimal, the set $L$ is a minimal left ideal, and hence subsemigroup, of $(\beta \N, +)$.  By \cite[Thm. 2.6]{hindman_strauss_book_2012}, the set $L$ is compact.  Therefore, the set $L$ is a compact subsemigroup of $(\beta \N, +)$ containing $p$.

    It follows from the proof of \cref{thm_equiv_forms_of_max_idempot_filter} that the set $\overline{\filter}$ is a compact subsemigroup of $(\beta \N, +)$.  It clearly contains $p$.

    Define $X = L \cap \overline{\filter}$.  Since both $L$ and $\overline{\filter}$ are compact subsemigroups of $(\beta \N, +)$ containing $p$, so is $X$.  By \cite[Thm. 2.5]{hindman_strauss_book_2012}, there exists an idempotent ultrafilter $q \in X$.  Since $X \subseteq L$, we see that $q$ is minimal.  Since $X \subseteq \overline{\filter}$, we see that $\filter \subseteq q$.  Thus, we have shown that $\filter$ is contained in a minimal, idempotent ultrafilter in $L$, as desired.
\end{proof}

\begin{theorem}
\label{cor_central_syndetic_is_strongly_central}
    Let $A \subseteq \N$ belong to a syndetic, idempotent filter.  For all minimal left ideals $L \subseteq \beta \N$, there exists an idempotent $p \in L$ such that $A \in p$.  In particular, for all thick sets $H \subseteq \N$, the set $A \cap H$ is central.
\end{theorem}

\begin{proof}
    Denote the syndetic, idempotent filter containing $A$ by $\family$. Let $L \subseteq \N$ be a minimal left ideal. By \cref{thm_sifs_are_contained_in_idempotents}, there exists an idempotent $p \in L$ such that $\family \subseteq p$.  In particular, we have that $A \in p$.

    If $H \subseteq \N$ is thick, by \cite[Thm. 4.48]{hindman_strauss_book_2012}, there exists a minimal left ideal $L \subseteq \overline{H}$.  By the reasoning in the previous paragraph, there exists an idempotent $p \in L$ containing $A$.  Therefore, $A \cap H \in p$, whereby the set $A \cap H$ is central.
\end{proof}

\begin{remark}
    In \cite{bergelson_hindman_strauss_2012}, a set $A \subseteq \N$ is called \emph{strongly central} if it has the property that for all minimal left ideals $L \subseteq \beta \N$, there exists an idempotent $p \in L$ containing $A$, and it is called \emph{very strongly central} if there exists a minimal system $(X,T)$, a nonempty, open set $U \subseteq X$, and a point $x \in \overline{U}$ such that $R(x,U) \subseteq A$.  It is a consequence of \cref{cor_central_syndetic_is_strongly_central} that members of syndetic, idempotent filters are strongly central.  In fact, a set is a member of a syndetic, idempotent filter if and only if it is very strongly central; this follows from \cref{intro_mainthm_full_chars_of_dcsyndetic}.  An example of a strongly central, but not very strongly central, set is given in \cite[Thm. 2.17]{bergelson_hindman_strauss_2012}.
\end{remark}

\subsection{Relations with established families and proof of \texorpdfstring{\cref{thm:relations_C_dPS_PS-intro}}{Theorem F}}

\label{sec_dcps_def_and_first_results}

In this subsection, we prove \cref{thm:relations_C_dPS_PS-intro} by combining simple properties of the families $\dcPS$ and $\dPS$ in Lemmas \ref{lemma_dcps_dps_are_pr} and \ref{lem:all_translates_are_dps}; abstract family algebra in Lemmas \ref{prop:general_abstract_intersection_filter_nonsense} and \ref{cor_shifted_class_equalities}; and the characterization of dynamically central syndetic sets in \cref{intro_mainthm_full_chars_of_dcsyndetic}.

\begin{lemma}
\label{lemma_dcps_dps_are_pr}
    The families $\dcPS$ and $\dPS$ are partition regular, and their duals, $\dcPS^*$ and $\dPS^*$, are filters.
\end{lemma}

\begin{proof}
    That the families $\dcPS$ and $\dPS$ are partition regular follows immediately by combining their definitions with \cref{lemma_class_cap_is_pr}.  It follows then from the discussion in \cref{sec_filter_pr_and_ufs} that the families $\dcPS^*$ and $\dPS^*$ are filters.
\end{proof}

\begin{lemma}
\label{lem:all_translates_are_dps}
If $A \in \dPS$, then for all $n \in \N$, $A - n$ and $A + n$ are also in $\dPS$.
\end{lemma}

\begin{proof}
Because the families $\dsyndetic$ and $\dthick$ are translation invariant (under positive and negative shifts, by \cref{lemma_translates_of_dsyndetic_sets,lemma_translates_of_dthick_sets}), it follows by \cref{lemma_idempotent_classcap_is_idempotent} that the family $\dPS = \dsyndetic \classcap \dthick$ is translation invariant.
\end{proof}

\begin{lemma}
\label{prop:general_abstract_intersection_filter_nonsense}
    If $\filter$ is a syndetic, idempotent filter and $\filtertwo$ is a $\central^*$, idempotent filter, then $\filter \familycap \filtertwo$ is a syndetic, idempotent filter.
\end{lemma}
\begin{proof}
    Since $\filter$ and $\filtertwo$ are filters, the family $\filter \familycap \filtertwo$ is a filter.  Thus, to see that it is a syndetic filter, it suffices to show that for all $A \in \filter$ and $B \in \filtertwo$, the set $A \cap B$ is syndetic.  Let $A \in \filter$, $B \in \filtertwo$, and $H \subseteq \N$ be thick.  Since the set $A$ belongs to a syndetic, idempotent filter, by \cref{cor_central_syndetic_is_strongly_central}, the set $A \cap H$ is central. Since $B \in \central^*$, the set $A \cap H \cap B$ is nonempty.  Because $H$ was an arbitrary thick set, we see that the set $A \cap B$ is syndetic, as desired. Finally, that $\filter \classcap \filtertwo$ is idempotent follows immediately from \cref{lemma_idempotent_classcap_is_idempotent} and the idempotency of $\filter$ and $\filtertwo$.
\end{proof}

\begin{lemma}
\label{cor_central_class_equalities}
\label{cor_shifted_class_equalities}
Let $\familytwo$ be a filter on $\N$.
\begin{enumerate}
    \item
    \label{item_central_star_class_cap}
    If $\filtertwo$ is a $\central^*$, idempotent filter, then $\filtertwo \subseteq \dcPS^*$.

    \item
    \label{item_syndetic_trans_inv_star_class_cap}
    If $\filtertwo$ is a syndetic, translation-invariant filter, then $\filtertwo \subseteq \dPS^*$.
\end{enumerate}
\end{lemma}

\begin{proof}
    \eqref{item_central_star_class_cap} \ Suppose that $\filtertwo$ is a $\central^*$, idempotent filter. To see that $\filtertwo \subseteq \dcPS^*$, it suffices by \cref{lemma_class_cap_is_pr} to show that for all $A \in \dcsyndetic$ and $B \in \filtertwo$, the set $A \cap B \in \dcsyndetic$.  Let $A \in \dcsyndetic$ and $B \in \filtertwo$. By \cref{intro_mainthm_full_chars_of_dcsyndetic}, the set $A$ belongs to a syndetic, idempotent filter $\filter$.  By \cref{prop:general_abstract_intersection_filter_nonsense}, the family $\filter \classcap \filtertwo$ is a syndetic, idempotent filter.  Since $A \cap B \in \filter \classcap \filtertwo$, we have by \cref{intro_mainthm_full_chars_of_dcsyndetic} that $A \cap B \in \dcsyndetic$, as desired.

    \eqref{item_syndetic_trans_inv_star_class_cap} \ Suppose that $\filtertwo$ is a syndetic, translation-invariant filter. By \cref{lemma_condition_on_subfamily_of_ps_star}, the family $\PS^*$ contains all syndetic, translation-invariant filters, so we have that $\filtertwo \subseteq \PS^* \subseteq \central^*$.  Note that $\filtertwo$ is idempotent because it is translation invariant.

    To see that $\filtertwo \subseteq \dPS^*$, it suffices by \cref{lemma_class_cap_is_pr} to show that for all $A \in \dsyndetic$ and $B \in \filtertwo$, the set $A \cap B \in \dsyndetic$.  Let $A \in \dsyndetic$ and $B \in \filtertwo$. By \cref{lemma_translates_of_dsyndetic_sets}, there exists $n \in \N$ such that $A - n \in \dcsyndetic$. Since $B - n \in \filtertwo - n \subseteq \filtertwo \subseteq \central^*$ and $\filtertwo$ is idempotent, it follows from \eqref{item_central_star_class_cap} and \cref{lemma_class_cap_is_pr} that
    \[
        (A \cap B) - n = (A - n) \cap (B - n) \in \dcsyndetic \classcap \filtertwo = \dcsyndetic.
    \]
    Adding $n$, by \cref{lemma_translates_of_dsyndetic_sets}, the set $A \cap B \in \dsyndetic$, as was to be shown.
\end{proof}

We are ready to prove \cref{thm:relations_C_dPS_PS-intro}.
\begin{proof}[Proof of \cref{thm:relations_C_dPS_PS-intro}]

By \cref{lem:central_characterization} and the fact that $\thick \subseteq \dthick \subseteq \dcthick$, we have that
    \begin{align*}
        \central = \dcsyndetic \classcap \thick &\subseteq \dcsyndetic \classcap \dcthick = \dcPS
    \end{align*}
    By \cref{rem:combinatorial_primes}, there is a set $A \in \dthick$ which is not an IP set. Since every central set is an IP set, we see that the set $A \in \dcPS$ but $A \not \in \central$. Thus, the containment $\central \subseteq \dcPS$ is proper.

 By \cref{lemma_condition_on_subfamily_of_ps_star}, the family $\PS^*$ is a syndetic, translation-invariant filter.  Since $\PS^* \subseteq \central^*$ and translation-invariance implies idempotency, we have that $\PS^*$ is also a $\central^*$, idempotent filter.  By \cref{cor_shifted_class_equalities}, we see that $\PS^* \subseteq \dcPS^*$ and $\PS^* \subseteq \dPS^*$.  It follows that $\dcPS \subseteq \PS$ and $\dPS \subseteq \PS$.
The piecewise syndetic set $2\N - 1$ is not $\dcPS$ and so the inclusion $\dcPS \subseteq \PS$ is proper.

To show that $\PS \subseteq \dPS$, we will show that $\dPS^*$ is a syndetic, translation-invariant filter and appeal to \cref{lemma_condition_on_subfamily_of_ps_star} to see that $\dPS^* \subseteq \PS^*$, whereby $\PS \subseteq \dPS$.  Translation invariance of $\dPS^*$ follows from the positive translation invariance of $\dPS$ shown in \cref{lem:all_translates_are_dps}. That $\dPS^*$ is a filter is shown in \cref{lemma_dcps_dps_are_pr}.  That every member of $\dPS^*$ is syndetic follows from the fact that $\thick \subseteq \dthick \subseteq \dPS$, whereby the family $\dPS^*$ is syndetic.

Combining all aforementioned inclusions, we obtain
\[
    \central \subsetneq \dcPS \subsetneq \dPS = \PS,
\]
as desired.
\end{proof}

\subsection{Applications of \texorpdfstring{\cref{thm:relations_C_dPS_PS-intro}}{Theorem F}}

\label{sec_dcps_and_set_recurrence}

In this subsection, we deduce a number of properties of dynamically piecewise syndetic sets from \cref{thm:relations_C_dPS_PS-intro}.

\subsubsection{Translates and dilates}

The following theorems describe how the families $\dcPS$ and $\dPS$ behave under translations and dilations.  While the conclusions are basic, we were not able to provide an argument for them that avoids the use of the relatively difficult facts in \cref{thm:relations_C_dPS_PS-intro}.

\begin{theorem}
\label{lemma_translates_of_dps_sets}
    Let $A \subseteq \N$ be dynamically piecewise syndetic. The sets
    \begin{align}
    \label{eqn_translates_of_dps_set_to_central}
        \big\{ n \in \N \ \big| \ A-n \text{ is central} \big\} \text{ and } \big\{ n \in \N \ \big| \ A+n \text{ is central} \big\}
    \end{align}
    are dynamically syndetic. Moreover, 
    \begin{align}
        \label{eqn_relationship_between_dps_and_c_classes}
        \dPS = \bigcup_{n \in \N} (\central - n) = \bigcup_{n \in \N} (\central + n).
    \end{align}   
\end{theorem}

\begin{proof}
    Let $A \in \dPS$.  By \cref{thm:relations_C_dPS_PS-intro}, the set $A$ is piecewise syndetic.  By \cref{lem:central_characterization}, there exists a minimal system $(X,T)$, a point $x \in X$, a nonempty, open set $U \subseteq X$, and a thick set $H \subseteq \N$ such that $A \supseteq R(x,U) \cap H$.  For all $n \in R(x,U)$, we see that the set $A - n \supseteq R(T^n x, U) \cap (H-n)$, which, by \cref{lem:central_characterization}, is central.  Thus, the set in \eqref{eqn_translates_of_dps_set_to_central} contains $R(x,U) \in \dsyndetic$, as was to be shown.

    To arrive at the same result when $A-n$ is replaced by $A+n$, begin by writing $A \supseteq R(x,U) \cap H$ in the same way, noting that by \cite[Lemma 3.1]{glasscock_le_2024}, we may assume that the system $(X,T)$ is invertible. We claim that the set $R_{T^{-1}}(x,U)$ is contained in the set in \eqref{eqn_translates_of_dps_set_to_central} with $A-n$ replaced by $A+n$.  Since the system $(X,T^{-1})$ is minimal (see, for example, \cite[Lemma 2.7]{glasscock_koutsogiannis_richter_2019}), this will demonstrate that the set in \eqref{eqn_translates_of_dps_set_to_central} is dynamically syndetic.
    
    Suppose $n \in R_{T^{-1}}(x,U)$ so that $T^{-n} x \in U$.  It is quick to check that $R(T^{-n}x,U) \cap \{n+1, n+2, \ldots\} \subseteq R(x,U) + n$.  Since $R(T^{-n}x,U) \in \dcsyndetic$, it follows by \cref{lemma_ds_dcs_modifyable_on_a_finite_set} that the set $R(T^{-n}x,U) \cap \{n+1, n+2, \ldots\} \in \dcsyndetic$. It follows by \cref{lem:central_characterization} that the set
    \[A + n \supseteq \big(R(x,U) + n\big) \cap \big(H + n\big) \supseteq \big(R(T^{-n}x,U) \cap \{n+1, n+2, \ldots\}\big) \cap \big(H + n\big),\]
    is central, as desired.

    Finally, that \eqref{eqn_relationship_between_dps_and_c_classes} holds follows immediately from \eqref{eqn_translates_of_dps_set_to_central} for $A-n$ and $A+n$.
\end{proof}

\begin{remark}
The conclusions of \cref{lemma_translates_of_dps_sets} formulated in terms of the family algebra developed in \cref{sec_combinatorics} read: $\dPS \subseteq \{\N\} + \dPS$ and $\dPS \subseteq \dsyndetic + \central$, respectively. In particular, it shows $\dPS \subseteq \dsyndetic + \dcPS$. In analogy to \cref{lemma_translates_of_dsyndetic_sets} \eqref{item_dcs_translates_of_dcs_are_dcs}, we would like to show that $\dcPS \subseteq \dcsyndetic + \dcPS$, that is, that if $A$ is dynamically central piecewise syndetic, then the set in \eqref{eqn_translates_of_dt_set} is dynamically central syndetic.  We were not able to show this.  It should be easier to determine whether or not the family $\dcPS$ is idempotent: $\dcPS \subseteq \dcPS + \dcPS$. See \cref{quest_is_dcps_idempotent} below for more discussion.
\end{remark}

\begin{theorem}
\label{lemma_dilates_of_dps_sets}
    Let $A \subseteq \N$ and $k \in \N$. If $A$ is dynamically central piecewise syndetic, then so are the sets $kA$ and $A/k$. If $A$ is dynamically piecewise syndetic, then so is the set $kA$.
\end{theorem}

\begin{proof}
    Suppose $A \in \dcPS$, and write $A = B \cap C$ where $B \in \dcsyndetic$ and $C \in \dcthick$. We see that $kA = kB \cap kC$ and $A/k = B/k \cap C/k$, both of which are $\dcPS$ sets in view of \cite[Lemma 3.4]{glasscock_le_2024} and \cref{lem:dilates_of_dT_dcT}.

    Suppose $A \in \dPS$. By \cref{thm:relations_C_dPS_PS-intro}, $A \in \PS$  and so $kA \in \PS = \dPS$ as desired.
\end{proof}

\subsubsection{Stability under non-piecewise syndetic changes and a proof of \cref{thm:remove_nonPS_from_dynthick}}
\label{sec_stability}

The next theorem -- from which \cref{thm:remove_nonPS_from_dynthick} immediately follows -- says that we can remove a non-piecewise syndetic set from a set in any family studied in this paper and get a set in the same family.

\begin{theorem}
\label{cor:largest_family_satisfies_cap}
    Let $E \subseteq \N$.  Any of the following conditions is equivalent to the set $E$ being not piecewise syndetic:
    \begin{enumerate}
        \item
        for all $A \in \dsyndetic$, the set $A \setminus E \in \dsyndetic$;

        \item \label{item:remove_nonps_dynamical_thick}
        for all $A \in \dthick$, the set $A \setminus E \in \dthick$;

        \item
        for all $A \in \PS$, the set $A \setminus E \in \PS$;

        \item for all $A \in \PS^*$, the set $A \setminus E \in \PS^*$.
    \end{enumerate}
    
    Moreover, if the set $A$ is not piecewise syndetic, then
    \begin{enumerate}
    \setcounter{enumi}{4}
        \item
        for all $A \in \dcsyndetic$, the set $A \setminus E \in \dcsyndetic$;

        \item \label{item:remove_nonps_pointwise_recurrence}
        for all $A \in \dcthick$, the set $A \setminus E \in \dcthick$;

        \item
        for all $A \in \dcPS$, the set $A \setminus E \in \dcPS$;

        \item for all $A \in \dcPS^*$, the set $A \setminus E \in \dcPS^*$.
    \end{enumerate}
\end{theorem}

\begin{proof}
    Note that $E \subseteq \N$ is not piecewise syndetic if and only if $\N \setminus E \in \PS^*$.
    By \cref{thm:relations_C_dPS_PS-intro}, we have that $\PS^* = \dPS^* \subseteq \dcPS^*$.  That the first four statements are equivalent to $\N \setminus E$ being a member of $\PS^*$ then follows immediately from \cref{lemma_class_cap_is_pr}.  If $\N \setminus E \in \PS^*$, then $\N \setminus E \in \dcPS^*$.  The final four statements follow since, by \cref{lemma_class_cap_is_pr}, they are all equivalent to $\N \setminus E$ being a member of $\dcPS^*$.
\end{proof}

\subsubsection{Polynomial multiple recurrence, Brauer configurations, and a proof of \cref{main_thm_brauer}}
\label{sec_sets_of_poly_rec_and_brauer}

The results in this section use the abstract set algebra in \cref{cor_central_class_equalities} together with \cref{thm:relations_C_dPS_PS-intro} to upgrade Theorems 5.3 and 5.5 in \cite{glasscock_le_2024}.  Similarly to the proofs of those theorems, denote by $\filtertwo$ the upward closure of the family of all subsets of $\N$ of the form
\begin{align}
    \label{eqn_poly_return_set}
    \polyret \big(\mu, T_i, p_{i,j}, E \big)_{\substack{i = 1, \ldots, k \\ j = 1, \ldots, \ell}} \defeq \left\{ n \in \N \ \middle | \ \mu\left( \bigcap_{j=1}^{\ell} \left( \prod_{i=1}^{k} T_i^{p_{i,j}(n)} \right)^{-1} E \right) > 0 \right\},
\end{align}
where $k, \ell \in \N$, the tuple $(X, \mu, T_1, \ldots, T_k)$ is an invertible, commuting probability measure preserving system, the set $E \subseteq X$ satisfies $\mu(E) > 0$, and $p_{i, j} \in \Q[n]$ satisfies $p_{i,j}(\Z) \subseteq \Z$ and $p_{i,j}(0) = 0$ for all $1 \leq i \leq k$, $1 \leq j \leq \ell$.  The family $\filtertwo$ is comprised of sets that contain the times of returns of a set of positive measure in a probability measure space under polynomial iterates of a finite collection of invertible, commuting measure preserving transformations.

\begin{theorem}
\label{thm_dcps_is_poly_mult_comm_rec}
    Every dynamically central piecewise syndetic set is a set of polynomial multiple measurable recurrence for commuting transformations.
\end{theorem}

\begin{proof}
First we show any $\dcPS$ set is a set of polynomial multiple measurable recurrence for \emph{invertible} commuting transformations, that is, that $\dcPS \subseteq \familytwo^*$.
By \cite[Lemma 5.2]{glasscock_le_2024}, the family $\familytwo$ is a $\IP^*$, idempotent filter and so is a $\central^*$, idempotent filter. Therefore, by \cref{cor_central_class_equalities}, we have that $\familytwo \subseteq \dcPS^*$.  The dual statement is $\dcPS \subseteq \familytwo^*$, as desired.

To remove the invertibility assumption, we just use the measurable natural extension (cf. \cite[Lemma 7.11]{Bergelson-McCutcheon-IPpolynomialSzemeredi}) as in the proof of \cite[Theorem C]{glasscock_le_2024}.
\end{proof}

Central sets are known to contain an abundance of combinatorial configurations (cf. \cite[Ch. 14]{hindman_strauss_book_2012}). The following result demonstrates that dynamically central piecewise syndetic sets also contain some combinatorial configurations.  (Recall from \cref{thm:relations_C_dPS_PS-intro} that central sets are dynamically central piecewise syndetic.)

\begin{theorem}
\label{thm:Brauer_and_dcPS}
    Let $A \subseteq \N$ be dynamically central piecewise syndetic. For all $k \in \N$ and all $p_1, p_2, \ldots, p_k \in \Q[x]$ with $p_i(\Z) \subseteq \Z$ and $p_i(0) = 0$, there exist $x, y \in \N$ such that
    \begin{align}
        \label{eqn_poly_brauer_in_dcps}
        x, y, x + p_1(y), \ldots, x + p_k(y) \in A.
    \end{align}
\end{theorem}

\begin{proof}
    By \cref{thm:relations_C_dPS_PS-intro}, the set $A$ has positive upper Banach density, i.e. $d^*(A) > 0$, because it is piecewise syndetic. By the Furstenberg Correspondence Principle (\cite[Thm. 1.1]{Furstenberg-ErgodicBehavior}, see also \cite[Thm. 3.2.5]{mccutcheon_1999}), there is a probability measure preserving system $(X, \mu, T)$ and $E \subseteq X$ with $\mu(E) \geq d^*(A) > 0$ such that for all $n_1, \ldots, n_k \in \Z$,
    \begin{align*}
        \mu \big(T^{-n_1} E \cap \cdots \cap T^{-n_k} E \big) \leq d^* \big( (A-n_1) \cap \cdots \cap (A-n_k) \big).
    \end{align*}
    By \cref{thm_dcps_is_poly_mult_comm_rec}, setting $p_0 \equiv 0$, there is $y \in A$ for which
    \[
        0 < \mu \big(E \cap T^{-p_1(y)} E \cap \cdots \cap T^{-p_k(y)} E \big) \leq d^* \big( A \cap (A-p_1(y)) \cap \cdots \cap (A-p_k(y)) \big).
    \]
    Therefore, there exist (many) $x \in A$ for which \eqref{eqn_poly_brauer_in_dcps} holds, as desired.
\end{proof}

\begin{proof}[Proof of \cref{main_thm_brauer}]

\cref{main_thm_brauer} follows from 
Theorems \ref{thm_dcps_is_poly_mult_comm_rec} and \ref{thm:Brauer_and_dcPS}.
\end{proof}

\section{Open questions}
\label{sec:open_questions}

The open questions and directions in this section are organized into two categories: those that concern sets of pointwise recurrence and dynamically thick sets, and those that concern dynamically piecewise syndetic sets.

\subsection{Dynamically thick sets and sets of pointwise recurrence}
\label{sec_dy_thick_sets}

From the definitions, we can see that if $A \subseteq \N$ is a dynamically thick set, then for all dynamically syndetic sets $B \subseteq \N$, there exists $n \in A$ such that the set $B-n$ is dynamically central syndetic.  Does the converse hold, that is, does this property characterize dynamical thickness?

\begin{question}
\label{quest_alt_char_of_dcthick}
    Is it true that a set $A \subseteq \N$ is dynamically thick if and only if for all $B \in \dsyndetic$, there exists $n \in A$ such that $B-n \in \dcsyndetic$?
\end{question}

If we denote by $\dsyndetic - \dcsyndetic$ the family of sets of the form $B - \dcsyndetic$ where $B \in \dsyndetic$, then \cref{quest_alt_char_of_dcthick} asks whether or not $\dthick = (\dsyndetic - \dcsyndetic)^*$.  The equivalent, dual form of this equality is $\dsyndetic = \dsyndetic - \dcsyndetic.$ Thus, an equivalent form of \cref{quest_alt_char_of_dcthick} asks whether or not a set $A \subseteq \N$ is dynamically syndetic if and only if there exists $B \in \dsyndetic$ such that $B - \dcsyndetic \subseteq A$.\\

It is a simple exercise to partition a thick set into a disjoint union of two thick sets.  The same task for dynamically thick sets or sets of pointwise recurrence does not appear to be so simple.

\begin{question}
\label{ques:partition_into_two_pointwise_recurrence}
    Can every dynamically thick set (resp. set of pointwise recurrence) be partitioned into two disjoint dynamically thick sets (resp. sets of pointwise recurrence)?
\end{question}

There is a positive answer to the analogue of \cref{ques:partition_into_two_pointwise_recurrence} for sets of topological recurrence (\cref{example:set-topological-partition}), sets of measurable recurrence (\cite{Forrest_recurrence-in-dynamical-systems}, \cref{example:set-measurable-partition}), sets of Bohr recurrence \cite[Prop. 1.4]{Le-interpolation-first}, and sets of pointwise recurrence for nilsystems \cite[Lemma 3.3]{Le_interpolation_for_nilsequences}. Though not explicitly written in these terms, the proofs of all of these results follow by demonstrating the $\sigma$-compactness of the dual families, as defined and described in \cref{sec:sigma-compact}.  Since the family of dynamically central syndetic sets is not $\sigma$-compact (\cref{theorem_dy_syndetic_is_not_sigma_compact}), a new strategy must be devised to answer \cref{ques:partition_into_two_pointwise_recurrence}.\\

We showed in \cref{cor_weak_structure_result} that all dynamically thick sets take the form $\bigcup_{B \in \family} (B \cap H_B)$, where $\family$ is a robustly syndetic collection of sets (recall the definition from \cref{sec_general_form_of_dthick}) and $(H_B)_{B \in \family}$ is a collection of thick sets.  This result would be of more use were we able to say more about the family $\family$.

\begin{question}
\label{quest_robust_synd_family}
    Is it true that for all dynamically thick sets $A \subseteq \N$, there exists a countable, robustly syndetic collection $\family$ of subsets of $\N$ and, for all $B \in \family$, a thick set $H_B \subseteq \N$, such that
    \begin{align}
        \label{eqn_enhanced_dt_form}
        A = \bigcup_{B \in \family} \big( B \cap H_B \big)?
    \end{align}
\end{question}

A positive answer to \cref{quest_robust_synd_family} would strengthen \cref{cor_weak_structure_result} and would open the door to applications.  As an example, we show in the following lemma how a positive answer could be used to give a positive answer to \cref{ques:partition_into_two_pointwise_recurrence}.

\begin{lemma}
\label{lemma_solution_to_partition_problem}
    If the answer to \cref{quest_robust_synd_family} is positive, then every dynamically thick set can be partitioned into two disjoint dynamically thick sets.
\end{lemma}

\begin{proof}
    Suppose the answer to \cref{quest_robust_synd_family} is positive.  Let $A \in \dthick$, and write the set $A$ in the form given in \eqref{eqn_enhanced_dt_form}. Enumerate $\family$ as $\{B_1, B_2, \ldots\}$. For ease of notation, for each $i \in \N$, write $H_i$ instead of $H_{B_i}$. By passing to subsets of the $H_i$'s, we may assume without loss of generality that $H_i = \bigcup_{j=1}^{\infty} I_{i, j}$ where $I_{i,j}$ are finite intervals in $\N$ satisfying $\lim_{j \to \infty} |I_{i,j}| = \infty$ and, for all $i, j \in \N$, $\max I_{i,j} < \min I_{i, j+1}$.

    It is a simple exercise to choose positive integers
    \[n^{(1)}_1 < n^{(1)}_2 < n^{(2)}_1 < n^{(1)}_3 < n^{(2)}_2 < n^{(3)}_1 < n^{(1)}_4 < n^{(2)}_3 < n^{(3)}_2 < n^{(4)}_1 < \cdots\]
    so that the intervals $I_{i, n^{(i)}_j}$, $i, j \in \N$, are disjoint.  For $i \in \N$, define two thick sets
    \[
        H^{(1)}_i = \bigcup_{j=1}^{\infty} I_{i, n^{(i)}_{2j-1}} \quad \text{ and } \quad  H^{(2)}_i = \bigcup_{j=1}^{\infty} I_{i, n^{(i)}_{2j}}.
    \]
    Let
    \[  
        A_1 = \bigcup_{i=1}^{\infty}  \big(B_i \cap H^{(1)}_i \big) \quad \text{ and } \quad A_2 = \bigcup_{i=1}^{\infty}  \big(B_i \cap H^{(2)}_i \big).
    \]
    By construction, the sets $A_1$ and $A_2$ are disjoint subsets of $A$.  By appending $A \setminus (A_1 \cup A_2)$ to $A_1$, we get that $A = A_1 \cup A_2$.  Moreover, both $A_1$ and $A_2$ have the form in \cref{cor_weak_structure_result}, and so they are dynamically thick.
\end{proof}

In the event of a negative answer to \cref{quest_robust_synd_family}, one could weaken the requirement on the collection $\family$ while keeping enough to salvage the reasoning in \cref{lemma_solution_to_partition_problem}.  Such a modification to \cref{quest_robust_synd_family} could be: \emph{Is there is a collection $\family$ with the property that $\{ 1_B \ | \ B \in \family\}$ is a $\sigma$-compact subset of $\{0,1\}^{\N}$ and, for all $B \in \family$, a thick set $H_B \subseteq \N$ such that the map $\family \to \{0,1\}^{\N}$ given by $B \mapsto H_B$ is continuous such that $A$ has the form in \eqref{eqn_enhanced_dt_form}?}

\subsection{Dynamically piecewise syndetic sets}
\label{sec_open_quests_dps_sets}

It was shown in \cref{thm:relations_C_dPS_PS-intro} that family of dynamically central piecewise syndetic sets sits between the families of central and piecewise syndetic sets:
\[\central \subseteq \dcPS \subseteq \PS.\]
Central sets are known to have an abundance of combinatorial configurations \cite[Ch. 14]{hindman_strauss_book_2012}. In \cref{thm:Brauer_and_dcPS}, we show that every $\dcPS$ set contains ``Brauer''-type polynomial configurations $x, y, x + p_1(y), \ldots, x + p_k(y)$. On the other hand, \cref{rem:combinatorial_primes} shows that $\dcPS$ sets (in fact, dynamically thick sets) need not be IP sets. This naturally begs the question: to what extent do $\dcPS$ sets contain finite and infinite combinatorial configurations?  The following is a specific example of a question along these lines.  Recall the definition of an IP$_0$ set from \cref{ex_ip_k_sets_are_compact}.

\begin{question}
\label{ques:pointwise_recurrence_IP_0}
    Must every dynamically central piecewise syndetic set be an IP$_0$ set?
\end{question}

It follows from \cref{thm:Brauer_and_dcPS} that $\dcPS$ sets are IP$_2$ sets -- they contain configurations of the form $\{x, y, x+y\}$ -- but we were not able to iterate in order to find higher-order finite sums.  A simpler variant of \cref{ques:pointwise_recurrence_IP_0} is: \emph{Must dynamically thick sets contain configurations of the form $\{x,y,z, x+y,x+z,y+z, x+y+z\}$?} The dual form of this line of questioning is formulated as Question 6.7 in \cite{glasscock_le_2024}.

A positive answer to \cref{ques:pointwise_recurrence_IP_0} would combine with deep results of Furstenberg and Katznelson to give more context to the result in \cref{thm_dcps_is_poly_mult_comm_rec}, which says that a general family of times of set returns in ergodic theory is $\dcPS^*$.
A corollary of a result of Furstenberg and Katznelson \cite[Thm. 10.3]{furstenberg_katznelson_1985} (see the discussion and derivation in \cite[Sec. 7.1]{bergelson_glasscock_2020}) gives that a narrower family of times of set returns is IP$_0^*$ and hence, by a positive answer to \cref{ques:pointwise_recurrence_IP_0}, $\dcPS^*$. Thus, Furstenberg and Katznelson's result would go part of the way toward explaining the result in \cref{thm_dcps_is_poly_mult_comm_rec}.  To recover the full result in \cref{thm_dcps_is_poly_mult_comm_rec} via this line of reasoning would require an answer to a long-standing open question: \emph{Can the results of Bergelson and McCutcheon (eg. \cite[Thm. 7.12]{Bergelson-McCutcheon-IPpolynomialSzemeredi}) be upgraded from IP$^*$ to IP$_0^*$?}\\

Finite sums structure is related to idempotency, which is addressed in following question.

\begin{question}
    \label{quest_is_dcps_idempotent}
    Is the family of dynamically central piecewise syndetic sets idempotent, that is,
    \[\dcPS \subseteq \dcPS + \dcPS?\]
\end{question}

Note that the families $\central$ and $\dcsyndetic$ are idempotent. Were the family $\dcthick$ idempotent (a question asked in an equivalent form in \cite[Question 6.5]{glasscock_le_2024}), then by \cref{lemma_idempotent_classcap_is_idempotent,lemma_translates_of_dsyndetic_sets}, we would have a positive answer to \cref{quest_is_dcps_idempotent}.

Combining \cref{thm:relations_C_dPS_PS-intro,lemma_translates_of_dps_sets}, we have that
\[\dcPS \subseteq \dPS \subseteq \dsyndetic + \dcPS \subseteq \dPS + \dcPS.\]
This is superficially close to demonstrating a positive answer to \cref{quest_is_dcps_idempotent}.  In fact, it is natural to guess that perhaps $\dcPS \subseteq \dcsyndetic + \dcPS$.  Note, however, that the proof of $\dcPS \subseteq \dPS$ relies on the shift-punch machinery in \cite{glasscock_le_2024}.  It would be interesting and potentially useful to find a proof of this ``simple-looking'' fact that avoids this heavy machinery.

\bibliographystyle{abbrv}
\bibliography{dynthick}

\bigskip
\bigskip
\footnotesize
\noindent
Daniel Glasscock\\
\textsc{University of Massachusetts Lowell}\par\nopagebreak
\noindent
\href{mailto:daniel_glasscock@uml.edu}
{\texttt{daniel{\_}glasscock@uml.edu}}

\bigskip
\noindent
Anh N. Le\\
\textsc{University of Denver}\par\nopagebreak
\noindent
\href{mailto:daniel_glasscock@uml.edu}
{\texttt{anh.n.le@du.edu}}

\end{document}